\documentclass[10pt]{amsart}

\usepackage{hyperref} 
\hypersetup{
    colorlinks=true,       
    linkcolor=blue,          
    citecolor=magenta,        
    filecolor=magenta,      
    urlcolor=cyan           
}

\usepackage{mathtools}
\mathtoolsset{showonlyrefs,showmanualtags}

\usepackage{amssymb,amsrefs}
\usepackage{amsmath}
\usepackage{graphicx}
\usepackage{tikz}
\usepackage{amsthm}
\usepackage{bm}

\def\be{\begin{equation}}
\def\ee{\end{equation}}

\newtheorem{thm}{Theorem}[section]
\newtheorem{corollary}[thm]{Corollary}

\newtheorem{lemma}[thm]{Lemma}
\newtheorem{proposition}[thm]{Proposition}

\newtheorem{conjecture}[thm]{Conjecture}

\theoremstyle{remark}
\newtheorem{remark}{Remark}[section]

\numberwithin{equation}{section}

\def\note#1
{\marginpar
{\tiny $\leftarrow$
\par
\hfuzz=20pt \hbadness=9000 \hyphenpenalty=-100 \exhyphenpenalty=-100
\pretolerance=-1 \tolerance=9999 \doublehyphendemerits=-100000
\finalhyphendemerits=-100000 \baselineskip=6pt
#1}\hfuzz=1pt}

\newcommand{\N}{{\mathbb N}}

\newcommand{\R}{{\mathbb R}}
\newcommand{\T}{{\mathbb T}}

\newcommand{\Z}{{\mathbb Z}}

\newcommand{\la}{\langle}
\newcommand{\ra}{\rangle}
\newcommand{\mbf}{\mathbf}

\begin{document}

\title[Averages along the squares]{Averages along the Square  Integers: 
$\ell^p$ improving  and  Sparse Inequalities} 

\author[Han] {Rui Han} 
\address{ School of Mathematics, Georgia Institute of Technology, Atlanta GA 30332, USA}
\email {rui.han@math.gatech.edu}
\thanks{RH: Research supported in part by National Science Foundation grant DMS-1800689.}

\author[Lacey] {Michael T. Lacey} 
\thanks{MTL: Research supported in part by grant National Science Foundation grant DMS-1600693, and by Australian Research Council grant DP160100153. }
\address{ School of Mathematics, Georgia Institute of Technology, Atlanta GA 30332, USA}
\email {lacey@math.gatech.edu}

\author[Yang]{ Fan Yang}
\address{ School of Mathematics, Georgia Institute of Technology, Atlanta GA 30332, USA}
\email {ffyangmath@gmail.com}
\thanks{FY: Research supported in part by AMS-Simons Travel grant 2019-2021.}

\begin{abstract}
Let $f\in \ell^2(\Z)$. Define the average of  $ f$ over the square integers  by 
\begin{align}
A_N f(x)&:=\frac{1}{N}\sum_{k=1}^N f(x+k^2), 
\end{align}
We show that $ A_N$ satisfies a local  scale-free $ \ell ^{p}$-improving estimate, for $ 3/2 < p \leq 2$: 
\begin{equation*}
 N ^{-2/p'} \lVert A_N f \rVert _{ \ell ^{p'}} \lesssim N ^{-2/p} \lVert f\rVert _{\ell ^{p}},  
\end{equation*}
provided $ f$ is supported in some interval of length $ N ^2 $, and $ p' =\frac{p} {p-1}$ is the conjugate index. 
The inequality above fails for $ 1< p < 3/2$. 
The maximal function 
$ A f = \sup _{N\geq 1} |A_Nf| $ satisfies a similar sparse bound. Novel weighted and vector valued inequalities for $ A$ follow.  
A critical step in the proof requires the control of a logarithmic average over $ q$ of a function $G(q,x)$ counting the number of square roots 
of $x$ mod $q$.  One requires an estimate uniform in $x$.  
\end{abstract}

\maketitle
\setcounter{tocdepth}{1} 
\tableofcontents

\section{Introduction} 
The investigation of $ L^{p}$ improving properties of averages formed over submanifolds has been under intensive investigation in Harmonic Analysis  since 
first results for spherical averages by Littman \cite{MR0358443} and Strichartz  \cite{MR0256219} in the early 1970's.  
Our focus here is on corresponding questions in the discrete setting, a much more recent topic for investigation.  
For averages over the square integers, we prove a  scale free $ \ell ^{p}$-improving estimate, one that is  sharp, up to the endpoint. 
We then establish sparse bounds for an associated maximal function.  The latter implies novel weighted and vector valued inequalities.  

Let $f\in \ell^2(\Z)$. Define the average over the square integers by 
\begin{align}
A_N f(x):=\frac{1}{N}\sum_{k=1}^N f(x+k^2).
\end{align}
For a function $f$ on $\Z$, and an interval $I\subset \Z$, define 
\begin{align}\label{def:p_norm}
\langle f\rangle_{I,p}:=\left(\frac{1}{|I|} \sum_{x\in I} |f(x)|^p\right)^{1/p}
\end{align}
to be the normalized $\ell^p$ norm on $I$.
Throughout the paper, if $I=[a,b]\cap \Z$, with $a,b\in \Z$, is an interval on $\Z$,  let $2I=[a, 2b-a+1]\cap \Z$ be the doubled interval (on the right-hand-side), let $3I=[2a-b-1,2b-a+1]$ be the tripled interval which has the same center as $I$.

The first theorem we prove is the following local, scale free, $\ell^p$ improving estimate for $A_N$. 
It is sharp in the index $ p$, and the only such result that is currently known.  

\begin{thm}\label{thm:improving}
For any $3/2< p \leq 2$, there is a constant $ C_p$ so that 
for any integer $N\geq 1$, and for any interval $I$ with length $N^2$, and any function $f$ supported on $2I$,  we have
\begin{align*}
\langle A_N f\rangle_{I,p'}\lesssim C_p \langle f\rangle_{2I,p}.
\end{align*}
Above $ p' = \frac{p} {p-1}$. The inequality above cannot hold for $ 1< p < 3/2$. 
\end{thm}

Let us define the maximal operator along the square integers:
\begin{align}\label{def:max_A}
Af(x):=\sup_{N\geq 1} |A_Nf|.
\end{align}
The $ \ell ^{p}$ bounds for this maximal function are a famous result of Bourgain \cites{MR916338}. 
We are interested in the sparse bounds,  a recently very active area of investigation.  
We call a collection of intervals $\mathcal{I}$ in $\Z$ {\it sparse} if there are sets $\{E_I:\, I\in \mathcal{I}\}$ which are pairwise disjoint, $E_I\subset I$ and satisfy $|E_I|>\frac{1}{4}|I|$.
The $(r,s)$-sparse form $\Lambda_{\mathcal{I},r,s}$, indexed by the sparse collection $\mathcal{I}$ is 
\begin{align*}
\Lambda_{\mathcal{I},r,s}(f,g)=\sum_{I\in \mathcal{I}} |I| \la f\ra_{2I,r}\, \la g\ra_{I,s}.
\end{align*} 
A sparse bound is a scale-invariant $\ell^p$ improving inequality.  
Our theorem is the following
\begin{thm}\label{thm:sparse}
Let $\mathbf{Z}$ be the triangle with three vertices $Z_1=(0,1)$, $Z_2=(1,0)$, $Z_3=(2/3, 2/3)$, see Figure \ref{f:Sparse_triangle}.  
For all $(1/p, 1/q)$ in the interior of $\mathbf{Z}$, with $f=\chi_F$, $g=\chi_G$, there holds
\begin{align*}
(Af, g)\lesssim \sup_{\mathcal{I}} \Lambda_{\mathcal{I}, p,q}(f,g).
\end{align*}
\end{thm}

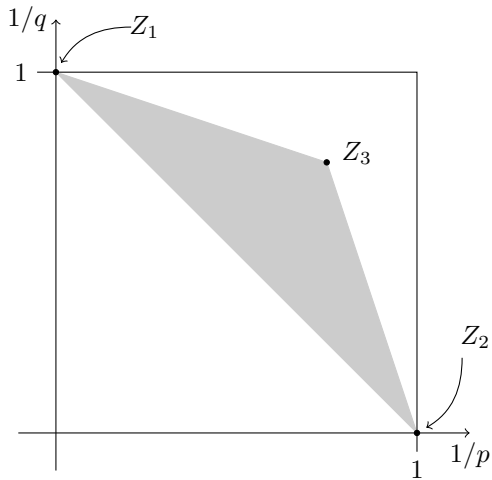
\begin{figure}
\begin{tikzpicture}
\draw [->]  (-.5, 0) -- (5.5,0) node[below] {$ 1/p$}; 
\draw [->]  (0,-.5) -- (0,5.5) node[left] {$ 1/q$}; 

\filldraw[black!20] (0, 4.8) -- (3.6, 3.6) -- (4.8,0) -- (0, 4.8);  

\filldraw (0, 4.8) circle (0.1em) node (Z1) {};
\filldraw (4.8, 0) circle (0.1em) node (Z2) {};
\filldraw (3.6,3.6) circle (0.1em) node (Z3) {};  

\draw (4, 4) node [below] {$Z_3$};

\draw (4.8,-.25 )  node [below] {$ 1$}  -- (4.8,4.8) -- (-.25,4.8) node [left] {$ 1$} ; 

\draw[->]    (1, 5.4) node {\ \ \ $Z_1$}   to [out = 180 , in = 60] (Z1) ; 
\draw[->]    (5.4, 1) node[above] {\ \ \ $Z_2$}   to [out = -90 , in = 30] (Z2) ; 
\end{tikzpicture}
\caption{Sparse bounds for the maximal function along the square integers 
hold for points $ (1/p, 1/q)$ in the interior of the triangle above. The points are $ Z_1 = (0,1) $, $ Z_2 = (2/3,2/3)$ and $ Z_3 = (1,0)$. } 
\label{f:Sparse_triangle}
\end{figure}

The interest in the sparse bound is that it immediately implies weighted and vector valued inequalities, which we return to in \S \ref {s:comp}. 
This is well documented in the literature. A sparse bound is the only known way to prove these types of estimates in the discrete setting.  

\bigskip 
Discrete Harmonic Analysis originates from the foundational work of Bourgain \cites{MR937581,MR937582, MR916338,MR1019960} on arithmetic ergodic theorems.  The essential element of these theorems are the maximal function inequalities for averages formed over polynomial sub-varieties 
of $ \mathbb Z ^{d}$.  This theory has been extended by several authors  \cites{MR1771530,MR2188130,MSW}.  Chief among these were 
 E.~M.~Stein and S.~Wainger. 
For a very recent, and deep, manifestation of this theory, we point to the recent papers \cites{MR3681393,2015arXiv151207518M,2018arXiv180309431K}.  
These references address many types of operators, including fractional integral operators \cites{MR2872554,MR1945293}. 
The latter operators are $ \ell ^{p} $ to $ \ell ^{q}$, but global and nature. The underlying difficulties behind these estimates are distinct from 
those of  scale free estimates. 

The scale free estimates were first studied for the discrete sphere by Hughes \cite{180409260H} and Kesler and Lacey \cite{180409845}. The analysis in this question hinges upon non-trivial bounds 
for Kloosterman sums.  The case of the spherical maximal function was addressed by Kesler \cites{180509925,180906468}. 
These papers reveal a remarkable parallel theory with the continuous case \cites{MR1388870,MR1432805,MR1949873}. 
In particular, the deepest aspects of these estimates depend upon Ramanujan sums.  Kesler's results were simplified and extended in 
\cite{181002240}. Discrete lacunary spherical bounds were proved in \cite{2018arXiv181012344K}. In sharp contrast to this paper, we do not know sharpness of any of the $ \ell ^{p}$ improving estimates in the case of the 
discrete sphere.

\bigskip 
We turn to the method of proof.  
Following the work of Bourgain \cites{MR937581,MR937582, MR916338,MR1019960}, we use the Hardy and Littlewood Circle method to make a detailed study of the corresponding multipliers.  
There are treatments of the Bourgain ergodic theorem on the square integers in the literature, but the methods used that we could find would not 
prove the sharp result.  
There is however a  very efficient version of Circle method for the square integers. This is established in an 
elegant paper of  Fiedler, Jurkat, K\"{o}rner \cite{MR0563894}, 
see Theorem~\ref{thm1:FJK} below. 

Using this important tool, we adapt another proof technique  of Bourgain \cite{MR812567}. 
The Fourier  multipliers associated to our operators 
are divided into several parts, each of which is either a `High Pass' or a `Low Pass' term. 
The High Pass terms are more elementary, in that one quantifies an $ \ell ^2 $-bound. 
The `Low Pass' terms are compared pointwise to the usual averages. This is the hard case. 
These  terms require a detailed analysis of certain exponential sums  related to the function 
\begin{equation*}
G (x,q) = \lvert  \{ \ell \in \mathbb Z/q \mathbb Z  \;:\;   \ell ^2 = x\}\rvert . 
\end{equation*}
See Lemma~\ref{lem:low_goal} for the precise function in question, as here we are taking small liberties for the sake of accessibility.  
It is always the case that $ G (x,q) \leq \sqrt q$.  However holding $ x$ fixed,  frequently in $q $, 
this function is only of the order of $ \log q$.  The  actual result is phrased in the language of logarithmic averages. 

The High Low  method is a common technique in the continuous setting \cite{MR1949873}.  Its appearance in the discrete setting 
is much more recent.  It was used (in the $ \ell^p $ to $ \ell^p$ setting) by Ionescu \cite{I}, and then Hughes \cite{MR3671577}. 
Its application to the setting of $ \ell ^{p}$ improving inequalities was initiated in \cites{181002240,180409845}.  Decompositions of the 
operators can involve several terms. For each, one only needs one estimate, High or Low.  

\bigskip 
The paper is organized as follows.  Well known results for Gauss sums are recalled in \S \ref{s:prelim}  followed by the 
two core initial estimates needed for the two main theorems above.  We then move to the proof of the uniform in scale estimate, 
namely Theorem~\ref{thm:improving}.  The core difficulty is the same in both Theorems, and is addressed in \S \ref{s:core}. 
We then turn to the sparse bound in \S \ref{s:sparse}.  Some complements, including open questions, are collected in \S \ref{s:comp}. 

\section{Preliminaries} \label{s:prelim}
\subsection{Notations}
Throughout the paper, let $e(x):=e^{2\pi i x}$.
Let 
\begin{align*}
\mathcal{F}_{\Z}(f)(\xi)=\sum_{x\in \Z} e(-\xi x) f(x),\ \ \xi\in \T=[0,1],
\end{align*}
be the Fourier transform on $\Z$,
and
\begin{align*}
\mathcal{F}_{\R}(f)(\xi)=\int_{\R} e(-\xi t) f(t)\, dt,\ \ \xi \in \R,
\end{align*}
be the Fourier transform on $\R$.
Define two normalized Gauss sums by 
\begin{align}  \label{e:G}
G(a,q)  & :=\frac{1}{q}\sum_{n=0}^{q-1} e(an^2/q).
\\ \label{e:G0}
G_0(a, q) & :=\frac{1}{2q} \sum_{n=0}^{2q-1} e(a n^2/2q)
\end{align} 
It is then clear that 
\begin{align}\label{eq:G0q=G2q}
G_0(a,q)=G(a,2q).
\end{align}
Define
\begin{align}\label{def:epsilon_m}
\varepsilon_m:=
\begin{cases}
1\qquad \text{if } m\equiv 1\, \mathrm{(mod} 4\mathrm{)}\\
i\, \qquad \text{if } m\equiv 3\, \mathrm{(mod} 4\mathrm{)}
\end{cases}
\end{align}
It is well-known that 
\begin{align}\label{eq:Gauss_G0}
G_0(a,q)=
\begin{cases}
0\qquad\qquad\qquad\qquad\qquad\quad\,\,\, \text{if } a\cdot q\, \text{ is odd}\\
q^{-1/2} \left(\frac{2a}{q}\right) e(\frac{(q-1)^2}{16})\,\,\,\,\,\,\,\qquad \text{if } 2|a\\
q^{-1/2} \left(\frac{q}{a}\right) e(\frac{a}{8})\qquad\qquad\qquad \text{if } 2|q
\end{cases}
\end{align}
where $\left(\frac{m}{n}\right)$ is the Jacobi symbol.
For $G(a,q)$, we have that for $(a,q)=1$, 
\begin{align}\label{eq:Gauss_G}
G(a,q)=
\begin{cases}
0\qquad\qquad\qquad\qquad\qquad\quad\,\, \text{if } q\equiv 2\, \mathrm{(mod} 2\mathrm{)}\\
\varepsilon_q q^{-1/2} \left(\frac{a}{q}\right)\qquad\qquad \, \, \qquad \text{if } q \text{ is odd}\\
(1+i)\varepsilon_a^{-1} q^{-1/2} \left(\frac{q}{a}\right)\, \, \,\,\,\,\, \qquad \text{if } a \text{ is odd and } 4|q
\end{cases}
\end{align}
When $(a,q)\neq 1$, we simply have
\begin{align}\label{eq:Gauss_G_not_coprime}
G(a,q)=G\left(\frac{a}{(a,q)}, \frac{q}{(a,q)}\right).
\end{align}
Clearly,
\begin{align}\label{eq:GaussG0_norm}
|G_0(a,q)|=
\begin{cases}
0\, \, \, \, \qquad\qquad \text{if } a\cdot q\, \text{ is odd}\\
q^{-1/2}\, \, \, \, \, \qquad  \text{otherwise}
\end{cases}
\end{align}

\subsection{The Core Estimates}  \label{s:core}
We state the core estimates to both of our main theorems.  
For $f, g\in \ell^2(\Z)$, we denote the standard inner product on $\ell^2(\Z)$ by $(f,g)$, namely
\begin{align}\label{def:inner_prod}
(f, g)=\sum_{x\in \Z} \overline{f(x)} g(x).
\end{align}

Since our goal is to prove Theorem \ref{thm:improving} for $p>3/2$, hence in an open range. It is sufficient to prove the following restricted weak type estimate.
\begin{thm}\label{thm:improving_weak}
For any $p>3/2$, for any interval $I$ with length $N^2$, we have
\begin{align*}
(A_N f, g) \lesssim_p \langle f\rangle_{2I,p} \langle g\rangle_{I,p}\, \lvert  I\rvert  
\end{align*}
holds for any indicator functions $f=\chi_F$ supported on $2I$ and $g=\chi_G$ supported on $I$.
\end{thm}
The core estimate of Theorem \ref{thm:improving_weak} is the following, where we decompose $A_N f$ into a High Pass and a  Low pass term. 
The High Pass term satisfies a very good $ \ell ^2 $ estimate, while the Low Pass term is compared to the usual averages, with a loss.  

\begin{lemma}\label{lem:High_Low_dec}
For any integer $J\in \{2^k: k\in \N\}$, we can decompose 
\begin{align*}
A_N f=H_{N,J}+L_{N,J},
\end{align*}
such that 
\begin{align} \label{e:loglog}
\begin{cases}
\la H_{N,J}\ra_{I,2}\lesssim J^{-1/2} \log J\, \la f\ra_{2I,2}\\
\la L_{N,J}\ra_{I,\infty}\lesssim J (\log J)^2\, \la f\ra_{2I,1}
\end{cases}
\end{align}
\end{lemma}
The proof of Lemma \ref{lem:High_Low_dec} is given in Section \ref{sec:lemma}.
We will now finish the proof of Theorem \ref{thm:improving_weak}.

\begin{proof}
Take $\varepsilon>0$ such that $p=3/(2-\varepsilon)$.
Lemma \ref{lem:High_Low_dec} clearly implies
\begin{align}
\la H_{N,J}\ra_{I,2}\lesssim_{\varepsilon} J^{-1/2+\varepsilon} \la f\ra_{2I,2}\, \text{ and }\, \la L_{N,J}\ra_{I,\infty}\lesssim_{\varepsilon} J^{1+\varepsilon}\, \la f\ra_{2I,1}.
\end{align}
We estimate
\begin{align}
|I|^{-1} (A_N f, g)
&\leq |I|^{-1} (H_{N,J}, g)+|I|^{-1} (L_{N,J}, g)\\
\label{e:loglog2}&\leq \la H_{N,J}\ra_{I,2}\, \la g\ra_{I,2}+\la L_{N,J}\ra_{I, \infty}\, \la g\ra_{I,1}\\
&\lesssim_{\varepsilon} J^{-1/2+\varepsilon} \la f\ra_{2I,2}\, \la g\ra_{I,2}+J^{1+\varepsilon} \la f\ra_{2I,1}\, \la g\ra_{I,1}.
\end{align}
Optimizing over $J$, clearly $J\sim \la f\ra_{2I,2}^{-2/3}\, \la g\ra_{I,2}^{-2/3}$. We have
\begin{align*}
|I|^{-1} (A_N f, g)\lesssim_{\varepsilon} \la f\ra_{2I,p}\, \la g\ra_{I,p},
\end{align*}
this proves Theorem \ref{thm:improving_weak}.  

\end{proof}

Turn to Theorem~\ref{thm:sparse}.  It suffices to prove the sparse bound restricting the supremum over $ N$ in \eqref{def:max_A} to powers of $ 2$.  
A sparse bound is typically proved by a recursive argument. 
To do this, we fix a large dyadic interval $E$, function $f=\chi_F$ supported on $2E$, and $g=\chi_G$ supported on $E$.
Let $C>0$ be a large absolute constant.
Consider a choice of stopping time $\tau: E\to \{1,...,\lfloor \sqrt{|E|}\rfloor \}\cap \{2^k, k\in \N\}$, 
so that the average $ A _{\tau (x)} f (x)$ is approximately maximal. 
We call $ \tau $ an {\it admissible stopping time} if for any subinterval $I\subset E$ with $\la f\ra_{3I,1}>C \la f\ra_{2E,1}$, 
we have $\min_{x\in I}\tau^2 (x)>|I|$.
The key recursive argument is the following:
\begin{lemma}\label{lem:sparse_key}
Let $(1/p, 1/q)$ be in the interior of $\mbf{Z}$.
Let $E, f, g$ be defined as above. 
For any admissible stopping time $\tau$, we have
\begin{align*}
(A_{\tau} f, g)\lesssim \la f\ra_{2E,p}\, \la g\ra_{E,q} \, |E|.
\end{align*}
\end{lemma}
Let us postpone the proof of this lemma, and finish the proof of Theorem \ref{thm:sparse} first. 

\begin{proof}[Proof of Theorem~\ref{thm:sparse}] 
We can assume there is a fixed dyadic interval $E$ such that $f=\chi_F$ is supported on $2E$ and $g=\chi_G$ is supported on $E$.
Let $\mathcal{I}_E$ be the maximal dyadic sub-intervals $I$ of $E$ for which $\la f\ra_{3I,1}>C\la f\ra_{2E,1}$. 
Then we have that for an appropriate choice of admissible $\tau$, 
\begin{align}\label{eq:max<tau+sparse}
(\sup_{N^2\leq |E|} A_N f, g)\leq (A_{\tau} f, g)+\sum_{I\in \mathcal{I}_E} (\sup_{N^2\leq |I|} A_N(f\chi_{2I}), g\chi_I)
\end{align}
By Lemma \ref{lem:sparse_key}, we can control the first term in \eqref{eq:max<tau+sparse},
\begin{align*}
(A_{\tau} f, g)\lesssim |E| \la f\ra_{2E,p}\, \la g\ra_{E,q}.
\end{align*}
For appropriate $C$, we have
\begin{align*}
\sum_{I\in \mathcal{I}_E}|I|\leq \tfrac{1}{4}|E|.
\end{align*}
We can recurse on the second term of \eqref{eq:max<tau+sparse} to construct our sparse bound.

\end{proof}

\section{Proof of Lemma \ref{lem:High_Low_dec}}\label{sec:lemma}

\subsection{The Initial Decomposition}
Our proof of Theorem \ref{thm:improving_weak} is built on a fine decomposition, using the Hardy-Littlewood Circle method, of the corresponding Fourier multiplier of $A_N$. 
Let
\begin{align}
K_N(x)=\frac{1}{N}\sum_{k=1}^N \delta_{-k^2} (x).
\end{align}
Thus $A_N f=f \ast  K_N$. The multiplier is a Weyl sum, given by 
$$\mathcal{F}_{\Z}K_N(\xi)=\frac{1}{N} \sum_{k=1}^N e(k^2 \xi).$$

Let $M=2^m \leq  N/4$, with $m\in \N$. 
This is the initial decomposition of the multiplier.  Write 
\begin{align}\label{eq:KN=aN+cN}
\mathcal{F}_{\Z}K_N(\xi)=a_N(\xi)+c_N(\xi),
\end{align}
where $a_N(\xi)$ is defined as follows:
\begin{align}\label{def:aN_gamma_eta}
\begin{cases}
a_N(\xi):=\sum_{s=1}^m a_{N,s}(\xi)\\
a_{N,s}(\xi):=\sum_{a/q\in \mathcal{R}_s} G_0(a,q) \eta_{2^{2s}}(2\xi-\frac{a}{q})\gamma_N(2\xi-\frac{a}{q})\\
\mathcal{R}_s:=\left\lbrace \text{reduced}\ a/q\in 2\T:\ 2^{s-1}\leq q<2^s\right\rbrace\\
\gamma_N(\xi):=\frac{1}{N} \int_0^N e(\xi t^2/2)\, dt\\
\eta_k(\xi):=\eta(k \xi),
\end{cases}
\end{align}
in which $\eta$ is a smooth bump function satisfying $\chi_{[-\frac{1}{4},\frac{1}{4}]}\leq \eta \leq \chi_{[-\frac{1}{2},\frac{1}{2}]}$.
We remark that the decomposition above depends upon $ J$, but we suppress the dependence in the notation.  
This decomposition, with  $M=J$ is needed for Lemma \ref{lem:High_Low_dec}, and with $M=N/4$ is needed for the maximal function sparse bounds.

The following estimate of $\gamma_N$ is known:
\begin{align}\label{eq:gammaN}
|\gamma_N(\xi)|\leq  \min \{ 1, \;  N^{-1} |\xi|^{-1/2} \}. 
\end{align}
We also note that 
\begin{align}\label{eq:gammaN_Fourier_inv}
\gamma_N(\xi)=\mathcal{F}_{\R}(h)(-N^2 \xi),
\end{align}
where $h(t)=\chi_{[0,1]}(t)\cdot \frac{1}{2\sqrt{t}}$.  This is the continuous version of the averages we are considering.  

Another useful fact is that for distinct $a_1/q_1, a_2/q_2\in \mathcal{R}_s$, we have
\begin{align}\label{eq:eta_supp_disjoint}
\mathrm{supp}\bigl(\eta_{2^{2s}}(\cdot -\tfrac{a_1}{q_1})\bigr) \cap \mathrm{supp}\bigl(\eta_{2^{2s}}(\cdot -\tfrac{a_2}{q_2})\bigr)=\emptyset.
\end{align}
The proof is trivial, just note that $|a_1/q_1-a_2/q_2|\geq 2^{-2s}$.


We will use the following results from Fiedler, Jurkat and  K\"{o}rner \cite{MR0563894}.
\begin{thm}\label{thm1:FJK} \cite{MR0563894}*{Thm. 1}  For all integers $ N$,  
\begin{align}\label{eq:FJK}
\mathcal{F}_{\Z}(K_N)(\xi)=\frac{g(a,q)}{N} \int_0^N e(r t^2/2q)\, dt+ \Omega,
\end{align}
in which
\begin{align}\label{eq:xi=a/q_FJK}
2\xi=\frac{a}{q}+\frac{r}{q},\ \ \ |r|\leq \frac{1}{4N},\ \ \ 0<q\leq 4N,\ \ \ (a,q)=1,
\end{align}
and
\begin{align*}
|\Omega|\leq C N^{-1}\sqrt{q},
\end{align*}
for some absolute constant $C>0$.
Here, see Theorem 5 of \cite{MR0563894},
\begin{align*}
g(a,q)=
\begin{cases}
0\qquad\qquad\qquad \text{if } a\cdot q \text{ is odd}\\
G_0(a,q) \qquad \,\,\,\,\,\, \text{otherwise}
\end{cases}
\end{align*}
\end{thm}
Note that the normalized Gauss sum satisfies $G_0(a,q)=0$ for $a\cdot q$ being odd, hence, $g(a,q)=G_0(a,q)$ always holds. 
Furthermore, adapting the integral in \eqref{eq:FJK} into our notation, we have
\begin{align*}
\frac{1}{N} \int_0^N e(rt^2/2q)\, dt=\gamma_N(2\xi-\tfrac{a}{q}).
\end{align*}
Hence \eqref{eq:FJK} turns into 
\begin{align}\label{eq:FJK--our}
\mathcal{F}_{\Z}(K_N)(\xi)=G_0(a,q) \gamma_N(2\xi-\tfrac{a}{q})+O(N^{-1} \sqrt{q}).
\end{align}
It holds whenever $\xi$ and $a/q$ satisfy \eqref{eq:xi=a/q_FJK}.

\subsection{The Estimate for $ c_N$}
This next lemma shows that we can take our first contribution to the High Pass term $H_{N,J}$ to be 
$ \mathcal F _{\mathbb Z } ^{-1} ( c_N \cdot \mathcal F _{\mathbb Z } f)$.

\begin{lemma}\label{lem:minor_arc} Let $c_N$ be defined as in \eqref{eq:KN=aN+cN}, it satisfies the estimate below uniformly in $M\leq N/4$.  
\begin{align} \label{e:minor_arc}
 \|c_N\|_{L^\infty}\lesssim M^{-1/2} \log M.
\end{align}
\end{lemma}

\begin{proof}
Recall that $ c_N = \mathcal F _{\mathbb Z } (K_N) - a_N$, and we need to estimate $ c_N (\xi ) $ for any $ \xi \in \mathbb T $.   
Dirichlet's theorem implies that for any $\xi$, there exists at least one reduced rational $a_*/q_*$ such that $1\leq q_*\leq 4N$ 
and $|2\xi-a_*/q_*|\leq 1/(4Nq_*)$. Let $s_*$ be defined as the unique number such that $a_*/q_*\in \mathcal{R}_{s_*}$.
Let us also note that $\xi$ and $a_*/q_*$ satisfy \eqref{eq:xi=a/q_FJK}.

We divide the discussion into two cases: 
(i). $s_*>m$.
(ii). $s_*\leq m$.

\smallskip \textit{Case (i).}
We estimate $\mathcal{F}_{\Z}(K_N)$ and $a_N$ separately.
For $\mathcal{F}_{\Z}(K_N)$, by \eqref{eq:FJK--our}, we have
\begin{align}\label{eq:case1_KN}
|\mathcal{F}_{\Z}(K_N)(\xi)|
&\lesssim |G_0(a_*, q_*)|\cdot |\gamma_N(2\xi-\frac{a_*}{q_*})|+N^{-1}\sqrt{q_*}
\\&\lesssim q_*^{-1/2}+N^{-1}\sqrt{q_*}\lesssim M^{-1/2},
\end{align}
where we have used the fact that $s_*>m$, hence $q_*\gtrsim M$, in the last line.
We also used the trivial estimate $\|\gamma_N\|_{L^\infty}\leq 1$.

Turning to $ a_N (\xi )$, we have 
\begin{align}
\lvert  a_N (\xi )\rvert & \leq \sum_{s=1} ^m \lvert  a _{N,s} (\xi )\rvert   . 
\\  \label{e:afar}
& \leq  \sum_{s=1} ^m \sum_{a/q\in \mathcal R_s} \lvert   G_0 (a,q) \rvert \cdot \lvert  \eta _{2 ^{2s}} (2 \xi - \tfrac{a}q ) \rvert \cdot \lvert \gamma _N (2 \xi - \tfrac a q) \rvert . 
\end{align}
For fixed $ \xi $ and $ s$ above, there is at most one $ a/q$ for which $ \eta _{2 ^{2s}} (2 \xi - \tfrac{a}q )  \neq 0$.  
And, for  any reduced $a/q \in \mathcal R_s $, we have 
\begin{equation} \label{e:qq}
\Bigl\lvert 2 \xi - \frac{a}q\Bigr\rvert \geq \Bigl\lvert \frac{a}q - \frac{a _{\ast }} {q _{\ast }}\Bigr\rvert 
\geq \frac{1} {q q _{\ast }} - \frac{1} {4N q _{\ast }} \gtrsim \frac1{q q _{\ast }} \gtrsim  \frac{1}{2^s q_{\ast}}
\end{equation}   
where we use $ q \leq M\leq N/4$.
Combine this estimate with the decay estimate  \eqref{eq:gammaN} on $ \gamma _N$ and the standard estimate on Gauss sums, to see that 
\begin{align*}
\eqref{e:afar} & \lesssim  N ^{-1} \sum_{s=1} ^{m} 2 ^{-s/2+s/2} \sqrt{q_{\ast}}   \lesssim  N^{-1/2} \log M\leq M^{-1/2} \log M,
\end{align*}
where we used $q_{\ast}\leq 4N$. This proves Case (i).

\smallskip\textit{Case (ii).} 
We estimate
\begin{align}
|\mathcal{F}_{\Z}(K_N)(\xi)-a_N(\xi)| & \leq 
\bigr\rvert G_0(a_*, q_*)\gamma_N(2\xi-\frac{a_*}{q_*})-a_{N,s_*}(\xi) \bigl\lvert 
\\& \quad + \label{e:ss}
\sum_{\substack{s=1\\ s\neq s_*}}^{m}|a_{N,s}(\xi)| + C N^{-1/2}.
\end{align}
The first term is zero.  Note that since $s_*\leq m$, we have $q_*\leq M\leq N/4$, hence
\begin{align*}
|2\xi-\frac{a_*}{q_*}|\leq \frac{1}{4N q_*}\leq \frac{1}{16 q_*^2}\leq \frac{1}{4} 2^{-2s_*},
\end{align*} 
which implies $\eta_{2^{2s_*}}(2\xi-\frac{a_*}{q_*})=1$.
Taking into account the disjointness of the supports of $\eta_{2^{2s_*}}$, see \eqref{eq:eta_supp_disjoint}, we have
\begin{align}\label{eq:case2_=}
G_0(a_*, q_*)\gamma_N(2\xi-\frac{a_*}{q_*})-a_{N,s_*}(\xi)=0.
\end{align}

For the term in \eqref{e:ss}, we argue in a manner similar to Case (i).  The inequality \eqref{e:qq} continues to hold, and we conclude 
in the same manner that 
\begin{align}\label{eq:case2_<>}
\sum_{\substack{s=1\\ s\neq s_*}}^{m}|a_{N,s}(\xi)|  \lesssim  M ^{-1/2} {\log M}   
\end{align}

Therefore, combining \eqref{eq:case2_<>} with \eqref{eq:case2_=}, we have
\begin{align*}
|\mathcal{F}_{\Z}(K_N)(\xi)-a_N(\xi)|\lesssim M^{-1/2}\log M.
\end{align*}
This proves the desired result. 
\end{proof}

\subsection{The Decomposition of $ a_N$}
In the rest of this section, we let $M=J=2^{s_0}$.
The multiplier $ a_N$ defined in \eqref{eq:KN=aN+cN} is further written as  $a_N = b _{N,1} + b _{N,2}$, where  
\begin{align}  \label{e:bN1def}
b_{N,1} & := \sum_{s=1} ^{s_0} \tilde{a}_{N,s}, 
\\
\label{def:tildea_Ns}
\tilde{a}_{N,s}(\xi)&:=\sum_{a/q\in \mathcal{R}_s} G_0(a,q) \eta_{qN^2/J} (2\xi-\frac{a}{q}) \gamma_N(2\xi-\frac{a}{q}).
\end{align}
There are two different properties needed.  The first is very easy. 

\begin{proposition}\label{p:bN2} We have the estimate 
\begin{equation}\label{e:bN2}
\lVert b _{N,2}\rVert _{ \ell ^{\infty} } \lesssim J^{-1/2} \log J.  
\end{equation}
\end{proposition}

\begin{proof}
The implicit definition of $ b _{N,2}  $ involves the differences $ \eta_{2^{2s}} (\theta ) - \eta_{qN^2/J}  (\theta )$. 
Observe that this difference is zero if $ \lvert  \theta \rvert < \frac{J} {4q N ^2 } $. Combine this with the Fourier decay 
estimate on $ \gamma _N$, \eqref{eq:gammaN}, to see that 
\begin{align*}
|(\eta_{2^{2s}}(2\xi-\frac{a}{q})- \eta_{qN^2/J} (2\xi-\frac{a}{q})) \gamma_N(2\xi-\frac{a}{q})|\lesssim q^{\frac{1}{2}} J^{-\frac{1}{2}}.
\end{align*}
Taking into account that $|G_0(a,q)|\leq q^{-\frac{1}{2}}$, we have
\begin{align}\label{eq:bN2_norm}
\|b_{N,2}\|_{\ell ^\infty}\lesssim J^{-1/2}\log{J}.
\end{align}
\end{proof}

The second estimate is at the core of the results of this paper. It is the Low Pass estimate below, and requires a sustained analysis to establish, 
which we take up in the next section. 

\begin{lemma}\label{l:bN1} 
For intervals $ I$ of length $ N ^2 $, and functions $ f$ supported on $ 2I$, there holds 
\begin{equation}\label{e:bN1}
\langle \mathcal F _{\mathbb Z } ^{-1} (b _{N,1}) \ast f  \rangle _{I, \infty } \lesssim J (\log J) ^2 \langle  f \rangle _{2I,1}.  
\end{equation}

\end{lemma}

\tikzset{
 treenode/.style = {shape=rectangle, rounded corners,
    draw, align=center,
    top color=white, bottom color=blue!10}}

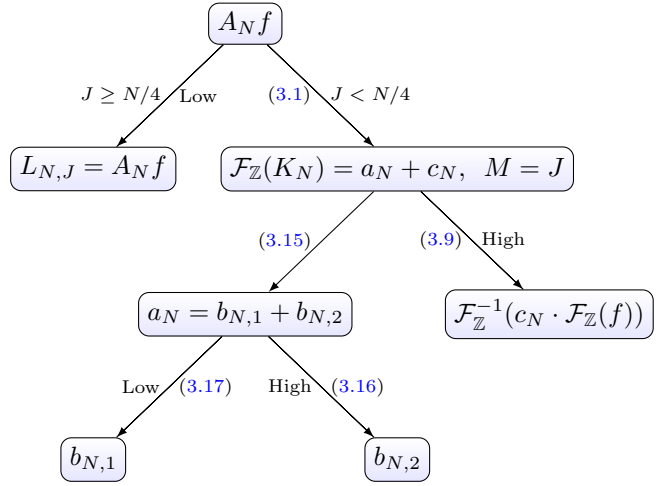
\begin{figure}
\begin{tikzpicture}[sibling distance=11.5 em, level distance=5.5em, 
    edge from parent/.style = {draw, -latex, font=\scriptsize}]
  \node [treenode] {$ \displaystyle  A _{N  } f $}
    child  { node [treenode]{$  L_{N,J} =  A _{N   } f  $} 
      edge from parent node [right]  { Low }
        edge from parent node [left]  {$ J\geq N/4$}  
                }
    child { node [treenode]{$  \displaystyle \mathcal{F}_{\Z}(K_ N )= a_ N  + c_ N $, { $M=J$}} 
      child { node [treenode]{ $\displaystyle a_ N  =   b _{N ,1} + b _{N ,2} $}
         child { node [treenode]{ $\displaystyle    b _{N ,1}$} 
           edge from parent node [left]  {Low} 
             edge from parent node [right]  {\eqref{e:bN1}}}
      child { node [treenode]{ $\displaystyle    b _{N ,2}$}   
              edge from parent node [left] {High} 
                edge from parent node [right]  {\eqref{e:bN2}}} 
            edge from parent node [left]  {\eqref{e:bN1def}}       
          }
      child { node [treenode] {$\displaystyle      \mathcal{F}_{\Z}^{-1}(c_{N}\cdot \mathcal{F}_{\Z}(f))   $} 
              edge from parent node [right]  {High} 
                edge from parent node [left]  {\eqref{e:minor_arc}}  } 
          edge from parent node [right]   {$ J < N/4$} 
            edge from parent node [left]  {\eqref{eq:KN=aN+cN}} };
\end{tikzpicture}

\caption{The High Low decomposition for Lemma~\ref{lem:High_Low_dec}. Compare to the more complicated decomposition in Figure~\ref{f:short}.}
\label{f:fixed}
\end{figure}

We have collected all the ingredients to complete the proof of our High Low decomposition. This argument is summarized in Figure~\ref{f:fixed}, 
as a point of comparison to the more complicated decomposition needed for the maximal function in Figure~\ref{f:short}. 

\begin{proof}[Proof of Lemma~\ref{lem:High_Low_dec}] 
Given integers $ N$ and $ J$, if $ J\geq N/4$, we set $ L _{N,J} = A_N f $, so that the High pass term is zero. 
Clearly, 
\begin{align}\label{eq:Low_J>N}
\la L_{N,J}\ra_{I,\infty}\leq \frac{1}{N}\sum_{x\in 2I} |f(x)|\lesssim N\la f\ra_{2I,1}\leq J\la f\ra_{2I,1}.
\end{align}
This proves the lemma in this case.  

The interesting case is $ J<N/4$. The Low pass term is given by $ b _{N,1}$ as defined in \eqref{e:bN1}.  
\begin{equation*}
L_{N,J}:=\mathcal{F}_{\Z}^{-1}( {b}_{N,1} ) \ast  f. 
\end{equation*}
By Lemma~\ref{l:bN1}, it satisfies the estimate required.  The High Pass term is then 
\begin{align}\label{def:High_Low}
H_{N,J}:=\mathcal{F}_{\Z}^{-1}(c_N ) \ast f  + \mathcal{F}_{\Z}^{-1}(b_{N,2} ) \ast f. 
\end{align}
By Lemma~\ref{lem:minor_arc} and Proposition~\ref{p:bN2}, this term satisfies the $ \ell ^2 $ estimate required of the 
High Pass term.  
\end{proof}

\section{The Low Pass Estimate} 
We give the proof of Lemma~\ref{l:bN1}, the core estimate of the proof.  
We will need these definitions.  
\begin{align}\label{def:H}
H(q,x) & :=\sum_{\substack{a=1 \\ (a,q)=1}}^{2q-1} G_0(a,q) e(ax/2q),
\\
H_0(q,x) & :=\sum_{a=0}^{q-1}G(a,q) e(ax/q),
\\
\textup{and} \qquad H_1(q,x) & :=\sum_{\substack{a=1\\ (a,q)=1}}^{q} G(a,q) e(ax/q).
\end{align}

The term to estimate is 
\begin{align}\label{eq:low_1}
(\mathcal{F}_{\Z}^{-1}(b_{N,1})* f)(x)=\sum_{y\in\Z}\, \sum_{q=1}^J H(q,x-y)\mathcal{F}_{\Z}^{-1} 
([\eta_{qN^2/J} (2\cdot) \gamma_N(2\cdot)]_{\mathrm{per}})(x-y) f(y),
\end{align}
where $[\eta_{qN^2/J} (2\cdot) \gamma_N(2\cdot)]_{\mathrm{per}}$ is obtained by extending $\eta_{qN^2/J} (2\cdot) \gamma_N(2\cdot)$ to a 1-periodic function.\linebreak
Obviously, $ \mathcal{F}_{\Z}^{-1} ([\eta_{qN^2/J} (2\cdot) \gamma_N(2\cdot)]_{\mathrm{per}})(x)=\mathcal{F}_{\mathcal{\R}}^{-1}(\eta_{qN^2/J} (2\cdot) \gamma_N(2\cdot))(x)$ for any $x\in \Z$. 
We have the following estimate
\begin{lemma}\label{lem:measure_conv}
\begin{align}
\|\mathcal{F}_{\mathcal{\R}}^{-1}(\eta_{qN^2/J} (2\cdot) \gamma_N(2\cdot))\|_{L^\infty}\lesssim \frac{J}{qN^2}.
\end{align}
\end{lemma}

\begin{proof}
We have, by \eqref{eq:gammaN_Fourier_inv}, that 
\begin{align}\label{eq:gammaN_inv}
\mathcal{F}_{\R}^{-1}(\gamma_N(2\cdot))(t)=\frac{1}{2N^2} h(-\frac{t}{2N^2}).
\end{align}
Hence 
\begin{align}\label{eq:h_L1}
\|\mathcal{F}_{\R}^{-1}(\gamma_N(2\cdot))\|_{L^1}\lesssim 1.
\end{align}
We also have
\begin{align}\label{eq:eta_qN2_inv}
\mathcal{F}_{\R}^{-1}(\eta_{qN^2/J}(2\cdot))(t)=\frac{J}{2qN^2} \mathcal{F}_{\R}^{-1}(\eta)(\frac{J}{2qN^2}t),
\end{align}
where $\mathcal{F}_{\R}^{-1}(\eta)$ is a Schwarz function. 
Hence 
\begin{align}\label{eq:eta_Linfty}
\|\mathcal{F}_{\R}^{-1}(\eta_{qN^2/J})\|_{L^\infty}\lesssim \frac{J}{qN^2}.
\end{align}
Combining \eqref{eq:h_L1} with \eqref{eq:eta_Linfty}, we have
\begin{align*}
\|\mathcal{F}_{\mathcal{\R}}^{-1}(\eta_{qN^2/J} (2\cdot) \gamma_N(2\cdot))\|_{L^\infty}\lesssim \frac{J}{qN^2},
\end{align*}
which is the desired result. 
\end{proof}

Therefore, by \eqref{eq:low_1} and Lemma \ref{lem:measure_conv}, we have
\begin{align*}
|(\mathcal{F}_{\Z}^{-1}(b_{N,1})* f)(x)|
&\lesssim \frac{J}{N^2} \sum_{y\in \Z}\, \sum_{q=1}^J \frac{|H(q,x-y)|}{q} |f(y)|
\\&
\lesssim J \Bigl\|\sum_{q=1}^J   \frac{ |H(q,\cdot)| } q  \Bigr\|_{\ell^{\infty}} \la f\ra_{2I,1}.
\end{align*}
 The required Low Pass estimate is a consequence of  the following

\begin{lemma}\label{lem:low_goal}
There exists an absolute constant $C>0$ such that
\begin{align*}
\Bigl\|\sum_{q=1}^J   \frac{ |H(q,\cdot)| } q  \Bigr\|_{\ell^{\infty}} \lesssim (\log J)^2. 
\end{align*}
 \end{lemma}

We remark that one can verify the  square root upper bound
$ |H(q,\cdot)| \lesssim \sqrt q $. This shows that the term above can be bounded by at most $ C \sqrt J \cdot \log J$. This yields a non-trivial 
$ \ell ^{p}$ improving estimate, but not the sharp estimate.   To verify the estimate above, it is 
essential that for fixed $ x$, the term $ \lvert  H (q,x)\rvert $ 
 can be as big as $ C\sqrt q$  for a few choice of $ q$.  
The rest of the section will be devoted to proving Lemma \ref{lem:low_goal}. 

\subsection{Preliminary Observations}

First, we do a few preliminary computations about $H(q,x), H_0(q,x)$ and $H_1(q,x)$.
\begin{lemma}\label{lem:H=H1}
For odd $q\geq 3$, we have
$H(q,x)=H_1(q,x)$. We also note $H(1,x)\equiv 0$, while $H_1(1,x)\equiv 1$.
\end{lemma}
\begin{proof}
The values of $H(1,x)$ and $H_1(1,x)$ can be computed from \eqref{eq:Gauss_G0} and \eqref{eq:Gauss_G}. 
We only need to prove the part for odd $q\geq 3$ now.
By \eqref{eq:Gauss_G0}, we have
\begin{align*}
H(q,x)
=&\sum_{\substack{a=1\\ (a,q)=1}}^{2q-1} G_0(a,q) e(ax/2q)   &   \eqref{def:H}\\
=&\sum_{\substack{a=1\\ (a,q)=1\\ a \text{ even}}}^{2q-1} G_0(a,q) e(ax/2q)    & \eqref{eq:Gauss_G0}\\
=&\sum_{\substack{a=1\\ (a,q)=1\\ a \text{ even}}}^{2q-1} G(a,2q) e(ax/2q)  & \eqref{eq:G0q=G2q}\\
=&\sum_{\substack{a=1\\ (a,q)=1}}^{q-1} G(2a, 2q) e(ax/q)
=H_1(q,x),
\end{align*}
where we used $G(2a,2q)=G(a,q)$ to obtain the last line. 
Let us observe that when $q=1$,
\begin{align*}
\sum_{\substack{a=1\\ (a,q)=1}}^{q-1}\neq \sum_{\substack{a=1\\ (a,q)=1}}^{q}.
\end{align*}
This is the reason why $H(1,x)$ and $H_1(1,x)$ take different values. 
\end{proof}

The function  $H_0(q,x)$ counts the number of square roots, as we see here. 
\begin{lemma}\label{lem:H_0}
\begin{align*}
H_0(q,x)=r_q(-x),
\end{align*}
in which $r_q(x)$ denotes the number of square roots $\ell$ of $x$ mod $q$, satisfying $0\leq \ell\leq q-1$.
\end{lemma}
\begin{proof}
This is a direct computation.
Indeed,
\begin{align*}
H_0(q,x)=&\frac{1}{q} \sum_{\ell=0}^{q-1} \sum_{a=0}^{q-1} e(a\ell^2/q) e(ax/q)\\
=&\frac{1}{q} \sum_{\ell=0}^{q-1} q\cdot \chi_{\ell^2+x\equiv 0 \text{ (mod } q\text{)}}\\
=&r_q(-x).
\end{align*}
This proves Lemma \ref{lem:H_0}.
\end{proof}

Let $QR(q)$ denote the quadratic residues of $q$ that are coprime to $q$.
It is well-known that for an odd prime number $p$, the following holds for any $k\geq 1$:
\begin{align}\label{eq:QR=QR}
x\in QR(p^k)\ \ \ \text{iff}\ \ \ x\in QR(p).
\end{align}

We show
\begin{lemma}\label{lem:rx}
Let $k\geq 1$.
Let $p$ be an odd prime.
Let $n\geq 0$ be such that $x=p^n x'$, where $(x', p)=1$. We have
\begin{align*}
r_{p^k}(x)=
\begin{cases}
p^{\lfloor \frac{k}{2} \rfloor}\qquad\ \  \text{if } n\geq k\\
2p^{\frac{n}{2}}\qquad\ \  \text{if } n \text{ is even satisfying }n<k, \text{ and } x' \in QR(p)\\
0\qquad\qquad \text{otherwise}
\end{cases}
\end{align*}
In particular, when $k=1$, we have
\begin{align}\label{eq:rx_<1}
|r_p(x)-1|\leq 
\begin{cases} 
1\qquad \text{ if } n=0\\
0\qquad \text{ if } n\geq 1
\end{cases}
\end{align}
\end{lemma}
\begin{proof}
The case when $n\geq k$ is easily checked.
If $n<k$ and $n$ is odd, we have $p^n \mid \ell^2$. Hence $p^{n+1}\mid \ell^2$, which forces $p^{n+1}\mid x$. This is impossible. 
If $n<k$ and $n$ is even. Let $x=p^n x'$ and $\ell=p^{n/2} \ell'$. We then have 
\begin{align*}
r_{p^k}(x)=p^{\frac{n}{2}} r_{p^{k-n}}(x').
\end{align*}
Note that $(x', p)=1$, hence we have
\begin{align*}
r_{p^{k-n}}(x')=
\begin{cases}
2\qquad \qquad \text{if } x' \in QR(p)\\
0\qquad \qquad \text{otherwise}
\end{cases}
\end{align*}
This proves Lemma \ref{lem:rx}.
\end{proof}

The next lemma is a simple consequence of the previous one. 
\begin{lemma}\label{lem:rx-rx}
For $k\geq 2$, we have
\begin{align*}
|r_{p^k}(x)-r_{p^{k-1}}(x)|=
\begin{cases}
p^{\frac{k}{2}}-p^{\frac{k}{2}-1}\qquad\qquad\, \text{if } k \text{ is even}, \text{and } p^k \mid x\\
p^{\lfloor \frac{k-1}{2} \rfloor}\qquad\qquad\qquad \text{if } x=p^{k-1} x' \text{ with } (x',p)=1\\
0 \qquad\qquad\qquad\qquad\, \text{otherwise}
\end{cases}
\end{align*}
\end{lemma}





\subsection{The Core of the Low Pass Estimate}

We quantify the fact that $H(q,x)$ is never more than $\sqrt q$, and can be large for only a few values of $x$.  
Lemmas, one for $ q$ odd and one for $ q$ even are stated here. 
\begin{lemma}\label{lem:odd_q_H_non_zero}
If $q$ is odd. Let $q=p_1^{k_1}\cdots p_m^{k_m}$ be its prime factorization.
Let 
\begin{align}\label{def:Doq}
\mathcal{D}_o(q):=\bigcap_{j=1}^m \bigl\{x\in \Z:\, \textup{ either ($ k_j $  is even and $p_j^{k_j}\mid x$),  or  ($ p_j^{k_j-1}\mid x $  but $ p_j^{k_j}\nmid x$)}  \bigr\}.
\end{align} 
Then we have 
\begin{align*}
|H(q,x)|\leq 
\begin{cases}
0\qquad\qquad\qquad\qquad\qquad \text{ if } x\notin \mathcal{D}_o(q)\\
\prod_{j=1}^m p_j^{\lfloor k_j/2\rfloor }\qquad\qquad\,\,\,\,\,\,\,\, \text{ if } x\in \mathcal{D}_o(q)
\end{cases}
\end{align*}
\end{lemma}

\begin{lemma}\label{lem:even_q_H_non_zero}
If $q$ is even. Let $q=2^b p_1^{k_1}\cdots p_m^{k_m}$ be its prime factorization. Let
\begin{align}\label{def:Deq}
\begin{split}
\mathcal{D}_e(q)&=
\{x\in \Z:\ 2^{\max(b-2,0)} \mid x\} 
\\ & \qquad 
\cap \bigcap_{j=1}^m \bigl\{x\in \Z:\, \textup{ either }  (\textup{$k_j$ is even and $p_j^{k_j}\mid x$}),\ \textup{ or }  (\textup{$p_j^{k_j-1}\mid x  $  but 
$ p_j^{k_j}\nmid x$}) \bigr\}.
\end{split}
\end{align}
Then we have 
\begin{align*}
|H(q,x)|\leq 
\begin{cases}
0\qquad\qquad\qquad\qquad\qquad \text{ if } x\notin \mathcal{D}_e(q)\\
2^{\frac{b}{2}} \prod_{j=1}^m p_j^{\lfloor k_j/2\rfloor}\qquad\qquad\, \text{ if } x\in \mathcal{D}_e(q)
\end{cases}
\end{align*}
\end{lemma}

These two lemmas imply the following, where we combine the cases of $ q$ odd and even. The first lemma treats $ x\neq 0$, the second $ x=0$.  
\begin{lemma}\label{lem:fix_x_vary_q}
Let $x=2^a p_1^{\ell_1}\cdots p_m^{\ell_m}$ be the prime factorization of $x$.
Let $\{p_j\}_{j=m+1}^{w}$ be the set of all the distinct prime numbers that are contained in $[1,J]$, which are different from $2, p_1,...,p_m$.
Let 
\begin{align}\label{def:Dx}
\mathcal{D}(x):=\{q\in [1, J]:\ &q=2^b p_1^{k_1}\cdots p_m^{k_m} p_{m+1}^{s_{m+1}}\cdots p_w^{s_w}, \,  0\leq b\leq a+2 ,\\
&\qquad 0\leq s_u\leq 1 \text{ for } m+1\leq u\leq w,\ \textup{ and } \mbf{k}_m:=(k_1,...,k_m)\in \mathcal{A}(x)\}, \notag
\end{align}
where 
\begin{align}\label{def:Ax}
\mathcal{A}(x)=\bigl\{\mbf{k}_m\in \Z^{m}:\ &\mbf{p}_m^{\mbf{k}_m}:=\prod_{j=1}^m p_j^{k_j}\leq J, \textup{ furthermore, each  $k_j$ satisfies}:\\ 
&\textup{either }  (\textup{$k_j$  is even and $0\leq k_j\leq \ell_j$} ), \textup{ or } k_j=\ell_j+1 \bigr\} \notag
\end{align} 
We have
\begin{itemize}
\item $H(q,x)=0$ for $q\notin \mathcal{D}(x)$. 
\item For each $q=2^b p_1^{k_1}\cdots p_m^{k_m} p_{m+1}^{s_{m+1}}\cdots p_w^{s_w}\in \mathcal{D}(x)$, 
there holds 
\begin{align}
|H(q,x)|\leq 
2^{\frac{b}{2}} \prod_{j=1}^m p_j^{\lfloor k_j/2\rfloor}.
\end{align}
\end{itemize}
\end{lemma}

\begin{lemma}\label{lem:fix_x=0}
Let $x=0$. Let $2,p_1,...,p_w$ be all the distinct primes numbers that are contained in $[1,J]$. 
Let 
\begin{align}\label{def:Dx_x=0}
\mathcal{D}(x):=\{q\in [1,J]:\ &q=2^b \prod_{j=1}^w p_j^{k_j}\, \, \text{ where } b\geq 0,  \text{ and } \mbf{k}_w\in \mathcal{A}(x)\},
\end{align}
where 
\begin{align}\label{def:Ax_x=0}
\mathcal{A}(x)=\{\mbf{k}_w\in \Z^{w}:\ \mbf{p}_w^{\mbf{k}_w}\leq J, \text{ and each } k_j \text{ is even}\}
\end{align}
We have 
\begin{itemize}
\item $H(q,x)=0$ for $q\notin \mathcal{D}(x)$. 
\item For each $q=2^b \prod_{j=1}^w p_j^{k_j}\in \mathcal{D}(x)$, we have
\begin{align}
|H(q,x)|\leq 
2^{\frac{b}{2}} \prod_{j=1}^w p_j^{k_j/2}.
\end{align}
\end{itemize}
\end{lemma}

We will postpone the proofs of Lemmas \ref{lem:odd_q_H_non_zero} and \ref{lem:even_q_H_non_zero}.
We instead finish the proof of Lemma \ref{lem:low_goal}, using Lemmas \ref{lem:fix_x_vary_q} and \ref{lem:fix_x=0}.
Indeed, the case $x=0$ is similar to (indeed, it is easier) the case $x\neq 0$, thus we only present the proof for $x\neq 0$ below.

\begin{proof}[Proof of  Lemma \ref{lem:low_goal}] 
We estimate
\begin{align*}
\sum_{q=1}^J \frac{|H(q,x)|}{q}=\sum_{q\in \mathcal{D}(x)} \frac{|H(q,x)|}{q}.
\end{align*}
Let $\mbf{s}_w:=(s_{m+1},...,s_w)$, $\mbf{p}_{w}^{\mbf{s}_w}:=p_{m+1}^{s_{m+1}}\cdots p_w^{s_w}$.
Let $\mathcal{C}:=\{\mbf{s}_w\in \{0,1\}^{w-m}:\ \mbf{p}_w^{\mbf{s}_w} \leq J\}$.
We have
\begin{align}\label{eq:final1}
\sum_{q\in \mathcal{D}(x)} \frac{|H(q,x)|}{q}
&\leq 
\sum_{b=0}^{a+2}\, 
\sum_{\mbf{k}_m\in \mathcal{A}(x)}\, \sum_{\mbf{s}_w\in \mathcal{C}}\,
\frac{\prod_{j=1}^m p_j^{\lfloor k_j/2\rfloor}}{2^{b/2} \prod_{j=1}^m p_j^{k_j}} \cdot \frac{1}{\mbf{p}_w^{\mbf{s}_w}} \notag\\
&\lesssim \left( \sum_{\mbf{k}_m\in \mathcal{A}(x)}\,
\frac{1}{\prod_{j=1}^m p_j^{k_j-\lfloor k_j/2\rfloor}}\right) \cdot \left(\sum_{\mbf{s}_w\in \mathcal{C}}\frac{1}{\mbf{p}_w^{\mbf{s}_w}}\right).
\end{align}
Note that for distinct $\mbf{s}_w$ and $\mbf{s}_w'$ belonging to $\mathcal{C}$, we have $\mbf{p}_w^{\mbf{s}_w}\neq \mbf{p}_w^{\mbf{s}_w'}$.
Hence
\begin{align}\label{eq:final2}
\sum_{\mbf{s}_w\in \mathcal{C}} \frac{1}{\mbf{p}_w^{\mbf{s}_w}}
\leq \sum_{n=1}^J \frac{1}{n}\lesssim \log J.
\end{align}
Let us also observe that if $k_j'$ and $k_j''$ are two distinct numbers belonging to $$\{k_j\in \Z: \textup{either }  (k_j \textup{ is even and } 0\leq k_j\leq \ell_j ), \textup{ or } k_j=\ell_j+1\},$$ then we have
\begin{align*}
k_j'-\lfloor\frac{k_j'}{2}\rfloor\neq k_j''-\lfloor\frac{k_j''}{2}\rfloor.
\end{align*}
This implies for distinct $\mbf{k}_m$ and $\mbf{k}_m'$ belonging to $\mathcal{A}(x)$, we have $\mbf{p}_m^{\mbf{k}_m}\neq \mbf{p}_m^{\mbf{k}_m'}$.
Hence 
\begin{align}\label{eq:final3}
 \sum_{\mbf{k}_m\in \mathcal{A}(x)}\,
\frac{1}{\prod_{j=1}^m p_j^{k_j-\lfloor k_j/2\rfloor}}\leq \sum_{n=1}^J \frac{1}{n}\lesssim \log J.
\end{align}
Combining \eqref{eq:final1}, \eqref{eq:final2} with \eqref{eq:final3}, we have
\begin{align*}
\sum_{q=1}^J \frac{|H(q,x)|}{q}\lesssim (\log J)^2.
\end{align*}
This proves the claimed result.
\end{proof}

Next, we prove Lemmas \ref{lem:odd_q_H_non_zero} and \ref{lem:even_q_H_non_zero}.
\subsubsection*{Proof of Lemma \ref{lem:odd_q_H_non_zero}}

The following multiplicative property of $H_1$ is proved in  Appendix \ref{app:mul}.
\begin{lemma}\label{lem:H_multi}
Let $q_1, q_2$ be two odd numbers that are coprime. Then we have
\begin{align}
|H_1(q,x)|=&|H_1(q_1,x)|\cdot |H_1(q_2,x)|.
\end{align}
\end{lemma}

Let $q=p_1^{k_1}\cdots p_m^{k_m}$ be its prime factorization. 
Lemmas \ref{lem:H=H1} and \ref{lem:H_multi} imply
\begin{align*}
|H(q,x)|=|H_1(q,x)|=\prod_{j=1}^m |H_1(p_j^{k_j},x)|.
\end{align*}
It then suffices to compute each $H_1(p_j^{k_j}, x)$. 
In general, let $p$ be an odd prime. 
We have that
\begin{align*}
H_1(p,x)=\sum_{\substack{a=1\\ (a,p)=1}}^p G(a,p)e(ax/p)=\sum_{a=0}^{p-1} G(a,p)e(ax/p)-1=H_0(p,x)-1=r_p(-x)-1,
\end{align*}
where we used Lemma \ref{lem:H_0}.
Hence by Lemma \ref{lem:rx}, we have
\begin{align}\label{eq:H1px}
|H_1(p,x)|\leq 
\begin{cases}
1 \qquad\qquad \text{if } (x,p)=1\\
0 \qquad\qquad \text{if } p\mid x
\end{cases}
\end{align}

For $k\geq 2$, we have
\begin{align*}
H_1(p^k,x)
=&\sum_{\substack{a=1\\ (a,p)=1}}^{p^k} G(a,p^k) e(ax/p^k)\\
=&\sum_{a=0}^{p^k-1} G(a,p^k) e(ax/p^k)-\sum_{a=0}^{p^{k-1}-1}G(a,p^{k-1})e(ax/p^{k-1})\\
=&H_0(p^k,x)-H_0(p^{k-1},x)\\
=&r_{p^k}(-x)-r_{p^{k-1}}(-x),
\end{align*}
where we have used Lemma \ref{lem:H_0} to obtain the last line.
Lemma \ref{lem:rx-rx} then implies
\begin{align}\label{eq:H1pkx}
|H_1(p^k,x)|\leq 
\begin{cases}
p^{\lfloor k/2\rfloor}& 
\textup{if } (\textup{$k$ is even and $p^k \mid x$}),   \textup{or } ( \textup{$ x=p^{k-1} x'$  with $(x',p)=1$}) \\
0&\text{otherwise}
\end{cases}
\end{align}
Here, we have assumed that $k\geq 2$.  But it also holds for $k=1$ by \eqref{eq:H1px}. That is, the inequality above 
 holds for any $k\geq 1$.
Therefore, Lemma \ref{lem:odd_q_H_non_zero} is justified.
\qed

\subsubsection*{Proof of Lemma \ref{lem:even_q_H_non_zero}}
This case requires a separate proof as  complications arise from the summing index $a$ below is in the bottom of the Jacobi symbol.  
Let $q=2^b p_1^{k_1}\cdots p_m^{k_m}=:2^b q'$ be the prime factorization of $q$.
We have
\begin{align*}
H(q,x)
=&\sum_{\substack{a=1\\ (a,q)=1}}^{2q-1} G_0(a,q)e(ax/2q) & \eqref{def:H}\\
=&\sum_{\substack{a=1\\ (a,q)=1}}^{2q-1} q^{-1/2} \left(\frac{2^b q'}{a}\right) e(a/8) e(ax/2q)  & \eqref{eq:Gauss_G0}\\
=&\sum_{\substack{a=1\\ (a,q)=1}}^{2q-1} q^{-1/2} \left(\frac{2}{a}\right)^b \left(\frac{q'}{a}\right) e(a/8) e(ax/2q)\\
=&\sum_{\substack{a=1\\ (a,q')=1 \\ a \text{ odd}}}^{2q-1} q^{-1/2} (-1)^{\frac{(a^2-1)b}{8}} \left(\frac{a}{q'}\right) (-1)^{\frac{(a-1)(q'-1)}{4}} e(a/8) e(ax/2q)
\end{align*}
Here, we have used the multiplicative property of the Jacobi symbol, and quadratic reciprocity. 
Let 
\begin{align*}
H_j(q,x):=\sum_{\substack{a=1,\\ (a,q')=1\\ a\equiv j \text{ (mod} 8 \text{)}}}^{2q-1} q^{-1/2} \left(\frac{a}{q'}\right) e(ax/2q)
\end{align*}
With these notations, we can write
\begin{align}\label{eq:H_in_H1357}
H(q,x)
=&e(1/8) H_1(q,x)+(-1)^{b+\frac{q'-1}{2}} e(3/8) H_3(q,x)+(-1)^b e(5/8) H_5(q,x)+(-1)^{\frac{q'-1}{2}} e(7/8) H_7(q,x) \notag\\
=&\begin{cases}
e(1/8) (H_1(q,x)-H_5(q,x))+(-1)^{\frac{q'-1}{2}}e(3/8) (H_3(q,x)-H_7(q,x))\qquad\qquad \text{ if } b \text{ is even}\\
\\
e(1/8) (H_1(q,x)+H_5(q,x))-(-1)^{\frac{q'-1}{2}}e(3/8) (H_3(q,x)+H_7(q,x))\qquad\qquad \text{ if } b \text{ is odd}
\end{cases}
\end{align}
It remains to examine the four terms of $H_1(q,x)\pm H_5(q,x)$  and $H_3(q,x)\pm H_7(q,x)$.
They in turn will be obtained as certain linear combinations of the function 
\begin{align}\label{def:tildeHqx}
\tilde{H}(q,x):=\sum_{j=0}^7 H_j(q,x)=\sum_{\substack{a=1\\ (a,q')=1}}^{2q-1} q^{-1/2} \left(\frac{a}{q'}\right) e(ax/2q).
\end{align}
We prove the following.
\begin{lemma}\label{lem:Htilde_qx} Let $ q=2^b p_1^{k_1}\cdots p_m^{k_m}$ have the same factorization as in Lemma~\ref{lem:even_q_H_non_zero}.  
Let $\tilde{\mathcal{D}}_e(q)$ be defined as
\begin{align}\label{def:tilde_Deq}
\tilde{\mathcal{D}}_e(q)&:=\{x\in \Z:\ 2^{b+1} \mid x\} 
\\
&\qquad \cap \bigcap_{j=1}^m  \bigl\{x\in \Z:\, \textup{ either ($k_j $ is even and $p_j^{k_j}\mid x$) 
 or ($ p_j^{k_j-1}\mid x $  but $p_j^{k_j}\nmid x$)} \bigr\}.
\end{align}
We have
\begin{align*}
|\tilde{H}(q,x)|\leq
\begin{cases}
0\qquad\qquad\qquad\qquad\,\,\,\,\, \text{if } x\notin \tilde{\mathcal{D}}_e(q)\\
2^{\frac{b}{2}+1} \prod_{j=1}^m p_j^{\lfloor k_j/2\rfloor}  \qquad \text{if } x\in \tilde{\mathcal{D}}_e(q)
\end{cases}
\end{align*}
\end{lemma}
\begin{proof}
We write $a=\ell q'+h$, then we have
\begin{align*}
\tilde{H}(q,x)
&=\sum_{\ell=0}^{2^{b+1}-1} \sum_{\substack{h=1\\ (h,q')=1}}^{q'-1} q^{-1/2} \left(\frac{h}{q'}\right) e((\ell q'+h)x/2q)\\
&=q^{-1/2} \sum_{\substack{h=1\\ (h,q')=1}}^{q'-1}  \left(\frac{h}{q'}\right) e(hx/2q) \sum_{\ell=0}^{2^{b+1}-1} e(\ell x/2^{b+1})
\end{align*}
Clearly, if $2^{b+1}\nmid x$, we simply have
\begin{align}\label{eq:tildeHqx_case1}
\tilde{H}(q,x)=0.
\end{align}
If $2^{b+1}\mid x$, we write $x=2^{b+1} x'$ and we have
\begin{align}\label{eq:tildeHqx_case2}
\tilde{H}(q,x)
=&q^{-1/2} 2^{b+1} \sum_{\substack{h=1\\ (h,q')=1}}^{q'-1}  \left(\frac{h}{q'}\right) e(hx'/q')\notag \\
=&\varepsilon_{q'}^{-1} 2^{\frac{b}{2}+1} \sum_{\substack{h=1\\ (h,q')=1}}^{q'-1}  G(h,q') e(hx'/q')\notag\\
=&\varepsilon_{q'}^{-1} 2^{\frac{b}{2}+1}  H_1(q', x')\notag\\
=&\varepsilon_{q'}^{-1} 2^{\frac{b}{2}+1} H(q',x'),
\end{align}
by Lemma \ref{lem:H=H1}.
Applying Lemma \ref{lem:odd_q_H_non_zero} to $H(q',x')$, and combining \eqref{eq:tildeHqx_case1} with \eqref{eq:tildeHqx_case2}, we finish the proof of Lemma \ref{lem:Htilde_qx}.
\end{proof}

Next, we will use $\tilde{H}(q,x)$ to compute $H(q,x)$.
Shifting $x$ by $q$ in \eqref{def:tildeHqx}, we have
\begin{align}\label{eq:tildeHq_x+q}
\tilde{H}(q,x+q)=\sum_{j=0}^3 H_{2j}(q,x)-\sum_{j=0}^3 H_{2j+1}(q,x),
\end{align}
where we used $H_{2j}(q,x+q)=H_{2j}(q,x)$ and $H_{2j+1}(q,x+q)=-H_{2j+1}(q,x+q)$ for any $0\leq j\leq 3$.
Hence 
\begin{align}\label{eq:sum_H_odd}
\sum_{j=0}^3 H_{2j+1}(q,x)=\frac{1}{2}(\tilde{H}(q,x)-\tilde{H}(q,x+q))
\end{align}

Shifting $x$ by $q/2$ in \eqref{eq:sum_H_odd}, we have
\begin{align}\label{eq:sum_H_odd_4}
\sum_{j=0}^3 e((2j+1)/4) H_{2j+1}(q,x)=\frac{1}{2}(\tilde{H}(q,x+\frac{q}{2})-\tilde{H}(q,x+\frac{3q}{2})),
\end{align}
where we used $H_{2j+1}(q,x+q/2)=e((2j+1)/4) H_{2j+1}(q,x)$ for any $0\leq j\leq 3$.
Combining \eqref{eq:sum_H_odd} with \eqref{eq:sum_H_odd_4}, we have
\begin{align}\label{eq:sum_H_mod4}
\begin{cases}
H_1(q,x)+H_5(q,x)=\frac{1}{4}(\tilde{H}(q,x)-\tilde{H}(q,x+q))+\frac{1}{4i}(\tilde{H}(q,x+\frac{q}{2})-\tilde{H}(q,x+\frac{3q}{2}))\\
\\
H_3(q,x)+H_7(q,x)=\frac{1}{4}(\tilde{H}(q,x)-\tilde{H}(q,x+q))-\frac{1}{4i}(\tilde{H}(q,x+\frac{q}{2})-\tilde{H}(q,x+\frac{3q}{2}))
\end{cases}
\end{align}
For odd $b$, we can already compute $H(q,x)$. Indeed, by \eqref{eq:H_in_H1357}, we have
\begin{align}\label{eq:H_odd_b}
|H(q,x)|
\leq |H_1(q,x)+H_5(q,x)|+|H_3(q,x)+H_7(q,x)| 
\leq \frac{1}{2} \sum_{\ell=0}^3 |\tilde{H}(q,x+\frac{\ell q}{2})|.
\end{align}

For even $b\geq 2$,
shifting $x$ by $q/4$ in \eqref{eq:sum_H_mod4}, we have
\begin{align}\label{eq:sum_H_mod8}
\begin{cases}
H_1(q,x)-H_5(q,x)=\frac{e(-1/8)}{4}(\tilde{H}(q,x+\frac{q}{4})-\tilde{H}(q,x+\frac{5q}{4}))+\frac{e(-1/8)}{4i}(\tilde{H}(q,x+\frac{3q}{4})-\tilde{H}(q,x+\frac{7q}{4}))\\
\\
H_3(q,x)-H_7(q,x)=\frac{e(-3/8)}{4}(\tilde{H}(q,x+\frac{q}{4})-\tilde{H}(q,x+\frac{5q}{4}))-\frac{e(-3/8)}{4i}(\tilde{H}(q,x+\frac{3q}{4})-\tilde{H}(q,x+\frac{7q}{4}))
\end{cases}
\end{align}
where we used $H_{2j+1}(q,x+q/4)=e((2j+1)/8) H_{2j+1}(q,x)$ for any $0\leq j\leq 3$.
One can compute $H(q,x)$ by plugging \eqref{eq:sum_H_mod8} into \eqref{eq:H_in_H1357}, we have
\begin{align}\label{eq:H_even_b}
|H(q,x)|
\leq |H_1(q,x)-H_5(q,x)|+|H_3(q,x)-H_7(q,x)| \leq \frac{1}{2} \sum_{\ell=0}^3 |\tilde{H}(q,x+\frac{(2\ell+1)q}{4})|.
\end{align}
By Lemma \ref{lem:Htilde_qx} and equations \eqref{eq:H_odd_b}, \eqref{eq:H_even_b}, we have
\begin{align*}
\{x\in \Z:\, H(q,x)\neq 0\}\subseteq 
\begin{cases}
\bigcup_{\ell=0}^3 (\tilde{\mathcal{D}}_e(q)-\frac{\ell q}{2})\qquad\qquad\,\,\,\,\,\,\,\,\,\,\,\, \text{if } b \text{ is odd}\\
\bigcup_{\ell=0}^3 (\tilde{\mathcal{D}}_e(q)-\frac{(2\ell+1)q}{4})\qquad\qquad \text{if } b \text{ is even}
\end{cases}
\end{align*}
Note that when $b$ is odd, the sets 
$\{\tilde{\mathcal{D}}_e(q)-\frac{\ell q}{2}\}_{\ell=0}^3$ (and similarly for $\{\tilde{\mathcal{D}}_e(q)-\frac{(2\ell+1)q}{4}\}_{\ell=0}^3$ when $b$ is even) are pairwise disjoint, and their union is contained in $\mathcal{D}_e(q)$, where $\mathcal{D}_e(q)$ is as in \eqref{def:Deq}.
Plugging the upper bounds for $\tilde{H}$ in Lemma \ref{lem:Htilde_qx} into equations \eqref{eq:H_odd_b} and \eqref{eq:H_even_b}, we conclude the proof of Lemma \ref{lem:even_q_H_non_zero}.
\qed

\section{Sparse bounds} \label{s:sparse}

The sparse bounds have been reduced to  Lemma \ref{lem:sparse_key}, which we prove here. 
In the statement of this lemma, recall that $ \mathbf Z$ is convex hull of $ Z_1 = (0,1)$, $ Z_2 = (1,0)$ and $ Z_3 = (2/3, 2/3)$. 
The sparse bounds at points $ (1/p, 1/p')$ correspond to maximal function inequalities, with the point $ Z_1$ being the trivial 
$ \ell ^{\infty }$ to $ \ell ^{\infty }$ bound for the maximal operator $ A$.  The bound for $ \ell ^{p} \to \ell ^{p}$, for $ p $ close to one 
is (a special case of) the arithmetic ergodic theorem of Bourgain \cite{MR1019960}.  
Thus it suffices to show the lemma holds at $(1/p, 1/p)$ for any $p\in (3/2, 2]$. An interpolation argument would enable us to cover all the parameters in the interior of $\mbf{Z}$.

\tikzset{
 treenode/.style = {shape=rectangle, rounded corners,
    draw, align=center,
    top color=white, bottom color=blue!10}}
    
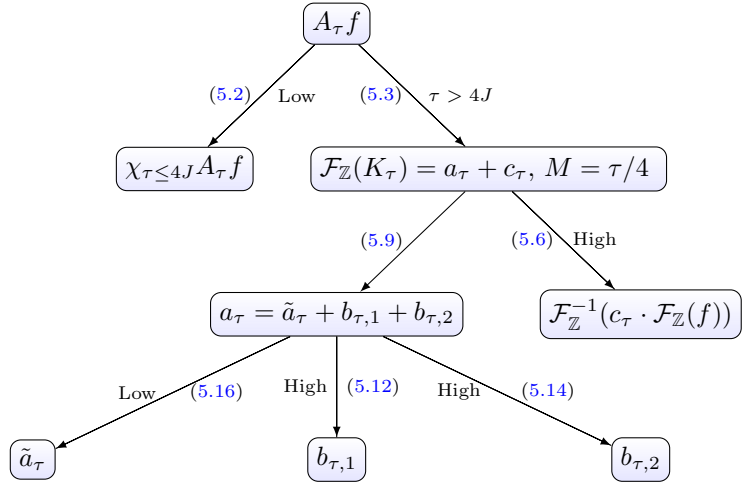
\begin{figure} 
\begin{tikzpicture}[sibling distance=11.5 em, level distance=5.5em, 
    edge from parent/.style = {draw, -latex, font=\scriptsize}]
  \node [treenode] {$ \displaystyle  A _{\tau  } f $}
    child  { node [treenode]{$ \displaystyle  \chi_{ \tau\leq  4J}  A _{\tau   } f  $} 
     edge from parent node [left] {\eqref{e:tauSmall}\ } 
     edge from parent node [right] {\ Low }}
    child { node [treenode]{$  \displaystyle \mathcal{F}_{\Z}(K_ \tau )= a_ \tau  + c_ \tau $,  $M=\tau/4$ } 
      child { node [treenode]{ $\displaystyle a_ \tau  = \tilde a _{\tau } + b _{\tau ,1} + b _{\tau ,2} $}
        child { node [treenode]{ $\displaystyle    \tilde a _{\tau }$} 
        edge from parent node [left] {Low\ \ }
         edge from parent node [right] {\ \eqref{e:tildea<} }} 
         child { node [treenode]{ $\displaystyle    b _{\tau ,1}$} 
         edge from parent node [left] {High} 
         edge from parent node [right] {\eqref{e:b1<}}}
      child { node [treenode]{ $\displaystyle    b _{\tau ,2}$}  
 		edge from parent node [left] {High\ \ } 
 		edge from parent node [right]  {\ \ \eqref{eq:btau2_1}}} 
         edge from parent node [left] {\eqref{e:a1}}}
      child { node [treenode] {$\displaystyle      \mathcal{F}_{\Z}^{-1}(c_{\tau}\cdot \mathcal{F}_{\Z}(f))   $}
       edge from parent node [right] {High} 
       edge from parent node [left] {\eqref{e:c<}}}
       edge from parent node [right] {\ $\tau >4J$}
        edge from parent node [left] {\eqref{e:K}\ }
         };
\end{tikzpicture}
\caption{The flow of the proof of  Lemma \ref{lem:sparse_key}.}  
\label{f:short}
\end{figure}

The situation is then  similar to that of the $\ell^p$-improving part, depending a High Low decomposition. 
Some additional complications  force a more elaborate decomposition, as detailed in  Figure \ref{f:short}.  
We introduce a parameter $J=2^{s_0}\in \{2^k:\, k\in \N\}$. 
We would like to decompose
\begin{align*}
A_\tau f=H_{\tau, J}+L_{\tau, J},
\end{align*}
such that 
\begin{align}\label{eq:HL_sparse}
\begin{cases}
\la H_{\tau, J}\ra_{E,2}\lesssim J^{-1/2} \log J\, \la f\ra_{2E,2}\\
\la L_{\tau, J}\ra_{E,\infty}\lesssim J (\log J)^2\, \la f\ra_{2E,1}
\end{cases}
\end{align}
Once proved,  we can argue as in the proof of the $\ell^p$-improving estimates, and show that for any $p>3/2$ we have
\begin{align*}
(A_{\tau} f, g)\lesssim |E| \la f\ra_{2E,p}\, \la g\ra_{E,p}.
\end{align*}
As we have remarked, this completes the proof of the Lemma.  

\bigskip 

The rest of the section will be devoted to proving \eqref{eq:HL_sparse}.
To this end, we decompose 
\begin{align*}
A_\tau f=\chi_{\tau\leq 4J} A_{\tau} f+\chi_{\tau>4J} A_{\tau} f.
\end{align*}
The part $\chi_{\tau\leq 4J} A_{\tau} f$ will be our first contribution to $L_{\tau,J}$. 
We have
\begin{lemma}\label{lem:tau<J}
The following holds 
\begin{align} \label{e:tauSmall}
 \la \chi_{\tau\leq 4J} A_{\tau} f\ra_{E,\infty}\lesssim J\, \la f\ra_{2E,1}.
\end{align}
\end{lemma}

\begin{proof}  
By the definition of admissibility, for any $x\in E$, we can find a good interval $I$ such that $x\in I$ and $\tau^2(x)=|I|$,
hence
\begin{align*}
A_{\tau(x)} f(x)\leq \frac{1}{\tau(x)} \sum_{k=1}^{\tau^2(x)} f(x+k)\lesssim \tau(x) \la f\ra_{3I,1}\leq J\la f\ra_{2E,1},
\end{align*}
where we used $\tau(x)\leq 4J$ in the last inequality. Since $I$ is a good interval, we have $\la f\ra_{3I,1}\lesssim \la f\ra_{2E,1}$, this finishes the proof.
\end{proof}

For the part $\chi_{\tau>4J} A_{\tau} f$, we will the decomposition in \eqref{eq:KN=aN+cN} and \eqref{def:aN_gamma_eta}. 
Recall that this is the initial decomposition $ \mathcal{F}_{\Z}(K_N)(\xi)=a_N(\xi)+c_N(\xi)$, where the dependence on $M$ was implicit 
in the notation. In our current situation, we apply \eqref{eq:KN=aN+cN} with $M=N/4=2^{s_1}$.  Then,  
\begin{align}
\mathcal{F}_{\Z}(K_N)(\xi)&=a_N(\xi)+c_N(\xi) \label{e:K}
\\ 
\label{e:aNN}
a_N(\xi) & :=\sum_{s=1}^{s_1} a_{N,s}(\xi), 
\\
\label{e:cN}
\| c_N \|  _ \infty & \lesssim   N ^{-1/2} \log N 
\end{align}
and $a_{N,s}$ is defined in \eqref{def:aN_gamma_eta}. The estimate \eqref{e:cN} follows from Lemma \ref{lem:minor_arc}, 
applied with $M=N/4$.

Our first contribution to the High Pass term $ H _{\tau ,J}$ is 
$\mathcal{F}_{\Z}^{-1}(\chi_{\tau>4J} c_{\tau} \cdot \mathcal{F}_{\Z}(f))$. 
\begin{lemma}\label{lem:ctau}
We have 
\begin{align} \label{e:c<}
\la \mathcal{F}_{\Z}^{-1}(\chi_{\tau>4J} c_{\tau} \cdot \mathcal{F}_{\Z}(f))\ra_{E,2}\lesssim J^{-1/2}\log J \, \la f\ra_{2E,2}.
\end{align}
\end{lemma}
\begin{proof}
Note that this is just an $\ell^2$ inequality, and we use a standard square function argument. We have
\begin{align}
\la \mathcal{F}_{\Z}^{-1}(\chi_{\tau>4J} c_{\tau}\cdot \mathcal{F}_{\Z}(f))\ra_{E,2} ^2 
\leq &\la \sup_{\text{dyadic } N>4J} |\mathcal{F}_{\Z}^{-1}(c_N\cdot \mathcal{F}_{\Z}(f))|\ra_{E,2} ^2  \notag\\
\leq & |E|^{-1} \sum_{\text{dyadic } N>J} \|\mathcal{F}_{\Z}^{-1}(c_N \cdot \mathcal{F}_{\Z}(f))\|_{\ell^2}^2  \label{eq:ctau_sq_fn}\\
\lesssim & \sum_{\text{dyadic } N>J} \|c_N\|_{L^{\infty}}^2   \la f\ra_{2E,2} ^2 , \label{eq:ctau_Par}
\end{align}
where we used square function to control the maximal function in \eqref{eq:ctau_sq_fn}, and we used Parseval's identity in \eqref{eq:ctau_Par}.
Applying \eqref{e:cN}, we have  
\begin{align*}
\sum_{\text{dyadic } N>J} \|c_N\|_{L^\infty}^2\lesssim \sum_{k>s_0} 2^{-k} k^2\lesssim J^{-1} (\log J)^2.
\end{align*}
Hence by \eqref{eq:ctau_Par}, we have the desired result. 
\begin{align*}
\la \mathcal{F}_{\Z}^{-1}(\chi_{\tau>4J} c_{\tau} \cdot  \mathcal{F}_{\Z}(f))\ra_{E,2}\lesssim J^{-1/2}\log J \, \la f\ra_{2E,2},
\end{align*}

\end{proof}

Next, we further decompose $a_{\tau}$, as given in \eqref{e:aNN}.   
Let 
\begin{align}\label{def:aNs_12}
\begin{cases}
a_{N,s}^{(1)}(\xi):=\sum_{a/q\in \mathcal{R}_s} G_0(a,q) \eta_{qN^2/J}(2\xi-\frac{a}{q}) \gamma_N(2\xi-\frac{a}{q})\\
a_{N,s}^{(2)}(\xi):=a_{N,s}(\xi)-a_{N,s}^{(1)}(\xi)
\end{cases}
\end{align}
We have, with the notation $ J = 2 ^{s_0}$,  
\begin{align}
\chi_{\tau>4J} a_{\tau}(\xi)=
&\chi_{\tau>4J} \sum_{s=1}^{s_0} a_{\tau,s}^{(1)}(\xi)    & (=:\tilde{a}_{\tau}(\xi))   \label{e:a1}\\
&+\chi_{\tau>4J} \sum_{s=1}^{s_0} a_{\tau,s}^{(2)}(\xi)   & (=:b_{\tau,1}(\xi)) \label{e:b1}\\
&+\chi_{\tau>4J} \sum_{s:\, J<2^s\leq \tau/4} a_{\tau,s}(\xi) &  (=:b_{\tau,2}(\xi)) \label{e:b2}. 
\end{align}
The terms $b_{\tau,1}$ and $b_{\tau,2}$ will be our second and third contributions to the High Pass term  $H_{\tau,J}$. 
The term $ \tilde a_ \tau $ will be a contribution to the Low Pass term.

\begin{lemma}\label{lem:btau1}
For the term $ b _{\tau ,1}$ defined in \eqref{e:b1}, we have 
\begin{align} \label{e:b1<}
\la \mathcal{F}_{\Z}^{-1}( b_{\tau, 1}\cdot  \mathcal{F}_{\Z}(f))\ra_{E,2}
\lesssim J^{-1/2}\log J \, \la f\ra_{2E,2}.
\end{align}
\end{lemma}
\begin{proof}
We apply Parseval's identity and a square function technique.
\begin{align}
\la \mathcal{F}_{\Z}^{-1}( b_{\tau, 1}\cdot  \mathcal{F}_{\Z}(f))\ra_{E,2}
\leq &\sum_{s=1}^{s_0} \la \mathcal{F}_{\Z}^{-1}( \chi_{\tau>4J}\, a_{\tau,s}^{(2)}\cdot  \mathcal{F}_{\Z}(f))\ra_{E,2} \notag\\
\leq &\sum_{s=1}^{s_0} \la \sup_{\text{dyadic } N>J} |\mathcal{F}_{\Z}^{-1}( a_{N,s}^{(2)}\cdot  \mathcal{F}_{\Z}(f))|\ra_{E,2} \notag\\
\lesssim &\sum_{s=1}^{s_0} \Bigl\| \sum_{\text{dyadic } N>J} |a_{N,s}^{(2)}|^2\Bigr\|_{L^\infty}^{\frac{1}{2}}\, \la f\ra_{2E,2} \label{eq:btau1_1}
\end{align}
It remains to estimate $\| \sum_{N>J} |a_{N,s}^{(2)}|^2\|_{L^\infty}$.
For any fixed $\xi$, let $a_0/q_0$ be uniquely determined by $\xi\in \mathrm{supp}(\eta_{2^{2s}}(\cdot-\frac{a_0}{q_0}))$.
Since $q\leq 2^s\leq J<N$, we have
\begin{align*}
\mathrm{supp}(\eta_{2^{2s}}(\cdot)-\eta_{qN^2/J}(\cdot))=[-\frac{1}{2^{2s+1}}, -\frac{J}{4qN^2}]\cup [\frac{J}{4qN^2}, \frac{1}{2^{2s+1}}].
\end{align*}
Let $N_0=2^{k_0}$ be the smallest dyadic number that is greater than $J$ and satisfies 
\begin{align*}
\xi\in \mathrm{supp}(\eta_{2^{2s}}(\cdot-\frac{a_0}{q_0})-\eta_{qN_0^2/J}(\cdot-\frac{a_0}{q_0})),
\end{align*}
thus $|2\xi-\frac{a_0}{q_0}|\geq J/(4qN_0^2)$.
Then for $k\geq k_0$, with $N=2^k$, we have that by \eqref{eq:gammaN},
\begin{align*}
|\gamma_N(2\xi-\frac{a_0}{q_0})|\lesssim N^{-1} \frac{\sqrt{q} N_0}{\sqrt{J}}.
\end{align*}
This implies, using the Gauss sum estimate \eqref{eq:GaussG0_norm},
\begin{align*}
|a_{N,s}^{(2)}(\xi)|\lesssim \frac{N_0}{N \sqrt{J}},
\end{align*}
uniformly in $\xi$.
Plugging the estimate above into \eqref{eq:btau1_1}, we have
\begin{align*}
\la \mathcal{F}_{\Z}^{-1}( b_{\tau, 1} \cdot \mathcal{F}_{\Z}(f))\ra_{E,2}
\lesssim J^{-1/2}\log J \, \la f\ra_{2E,2}.
\end{align*}
This proves Lemma \ref{lem:btau1}.
\end{proof}

\begin{lemma}\label{lem:btau2}
For the term $ b _{\tau ,2}$ defined in \eqref{e:b2}, we have 
\begin{align}\label{eq:btau2_1}
\la \mathcal{F}_{\Z}^{-1}( b_{\tau, 2}\cdot  \mathcal{F}_{\Z}(f))\ra_{E,2}
\lesssim J^{-1/2}\log J \, \la f\ra_{2E,2}.
\end{align}
\end{lemma}
\begin{proof}
The proof of this part crucially uses  Bourgain's multi-frequency maximal theorem, one of the main results of \cite{MR1019960}. 
The following is a corollary of that result, and the standard Gauss sum estimate.  
\begin{thm}\label{thm:BenT33}
For any $s\geq 1$, the following inequality holds
\begin{align} \label{e:multiF}
\|\sup_{N\geq 2^s} |\mathcal{F}_{\Z}^{-1} (a_{N,s} \cdot  \mathcal{F}_{\Z}(f))|\|_{\ell^2} \lesssim s2^{-\frac{s}{2}} \|f\|_{\ell^2}.
\end{align}
\end{thm}
This particular implies
\begin{align}\label{eq:btau2_2}
\la \sup_{N\geq 2^s} |\mathcal{F}_{\Z}^{-1} (a_{N,s}\cdot  \mathcal{F}_{\Z}(f)) |\ra_{E,2}\lesssim s2^{-\frac{s}{2}} \la f\ra_{2E,2}.
\end{align}
By triangle inequality, we have
\begin{align*}
\la \mathcal{F}_{\Z}^{-1}( b_{\tau, 2}\cdot  \mathcal{F}_{\Z}(f))\ra_{E,2}
\lesssim \sum_{2^s>J} \la \sup_{N\geq 2^s} |\mathcal{F}_{\Z}^{-1} (a_{N,s}\cdot  \mathcal{F}_{\Z}(f))\ra_{E,2}
\lesssim \sum_{2^s>J} s 2^{-\frac{s}{2}} \la f\ra_{2E,2}
\lesssim J^{-1/2} \log J\, \la f\ra_{2E,2}.
\end{align*}
This proves Lemma \ref{lem:btau2}.
\end{proof}

\begin{remark}\label{r:} The paper of Bourgain \cite{MR1019960} 
proves \eqref{e:multiF} with an estimate of the form $ s ^2 2 ^{- \frac{s}2} \lVert f\rVert _{\ell ^2 }$ on the right. 
That is the logarithmic term $ s$ is squared.  It is known that the estimate above holds.  
See for instance \cite{2018arXiv180309431K}*{Prop. 5.11}.
\end{remark}

Let 
\begin{align}\label{def:HLtau}
\begin{cases}
H_{\tau,J}:=\mathcal{F}_{\Z}^{-1}((\chi_{\tau>4J}c_\tau+b_{\tau,1}+b_{\tau,2})\mathcal{F}_{\Z}(f))\\
L_{\tau,J}:=\chi_{\tau\leq 4J} A_\tau f+\chi_{\tau\geq 4J}  \mathcal{F}_{\Z}^{-1} (\tilde{a}_{\tau}\cdot  \mathcal{F}_{\Z}(f))
\end{cases}
\end{align}
Combining Lemmas \ref{lem:ctau}, \ref{lem:btau1} and \ref{lem:btau2}, we have
\begin{align*}
\la H_{\tau,J}\ra_{E,2}\lesssim J^{-1/2} \log J \, \la f\ra_{2E,2}.
\end{align*}
This proves the desired estimate for the High Pass term in \eqref{eq:HL_sparse}.
In view of Lemma \ref{lem:tau<J}, to prove the estimate for the Low Pass term, it suffices to show the following.

\begin{lemma}\label{lem:tilde_a_tau_Low}
Under the assumption that $\tau > J$ pointwise, we have 
\begin{align} \label{e:tildea<}
\la \mathcal{F}_{\Z}^{-1}( \tilde{a}_{\tau} \cdot  \mathcal{F}_{\Z}(f))\ra_{E,\infty}\lesssim J (\log J)^2 \, \la f\ra_{2E,1}.
\end{align}
\end{lemma}

Indeed, this estimate is at the core of the sparse bound.  We need this preparation. 
 \begin{lemma}\label{lem:sparse_measure_con}
The following holds
\begin{align*}
|\mathcal{F}_{\R}^{-1} 
(\eta_{q\tau^2/J} (2\cdot) \gamma_\tau(2\cdot))(y)|\lesssim
\begin{cases}
\frac{J}{q\tau^2}\qquad\qquad \text{if } |y|\leq 4N^2\\
\\
\frac{q\tau^2}{Jy^2}\qquad\qquad \text{if } |y|>4N^2
\end{cases}
\end{align*}
\end{lemma}
\begin{proof}
Using \eqref{eq:gammaN_inv} and \eqref{eq:eta_qN2_inv}, we have
\begin{align*}
\mathcal{F}_{\R}^{-1} (\eta_{q\tau^2/J} (2\cdot) \gamma_\tau(2\cdot))(y)
=&\frac{J}{4q\tau^4} \int_{\R} \mathcal{F}_{\R}^{-1}(\eta)(\frac{J}{2q\tau^2} (y-z))h(-\frac{z}{2\tau^2})\, dz\\
=&\frac{J}{2q\tau^2} \int_0^1 \mathcal{F}_{\R}^{-1}(\eta)(\frac{J}{2q\tau^2}(y+2\tau^2 z^2))\, dz,
\end{align*}
where we used $h(z)=\chi_{[0,1]}(z)\cdot \frac{1}{2\sqrt{z}}$.

For $|y|\leq 4\tau^2$, using $|\mathcal{F}_{\R}(\eta)(z)|\lesssim 1$, we have
\begin{align*}
|\mathcal{F}_{\R}^{-1} (\eta_{q\tau^2/J} (2\cdot) \gamma_\tau(2\cdot))(y)|\lesssim \frac{J}{q\tau^2}.
\end{align*}
For $|y|>4\tau^2$, using $|\mathcal{F}_{\R}(\eta)(z)|\lesssim |z|^{-2}$, we have
\begin{align*}
|\mathcal{F}_{\R}^{-1} (\eta_{q\tau^2/J} (2\cdot) \gamma_\tau(2\cdot))(y)|\lesssim \frac{q\tau^2}{J} |y+2\tau^2z^2|^{-2}\lesssim \frac{q\tau^2}{J y^2}.
\end{align*}
Hence Lemma \ref{lem:sparse_measure_con} is proved.
\end{proof}

\begin{proof}[Proof of Lemma~\ref{lem:tilde_a_tau_Low}] 

For any fixed $x\in E$, $\tau$ is also fixed.
By \eqref{eq:low_1}, 
\begin{align}\label{eq:low_2}
(\mathcal{F}_{\Z}^{-1}(\tilde{a}_{\tau})* f)(x)=\sum_{y\in\Z}\, \sum_{q=1}^J H(q,y)\mathcal{F}_{\R}^{-1} 
(\eta_{q\tau^2/J} (2\cdot) \gamma_\tau(2\cdot))(y) f(x-y).
\end{align} 

Applying Lemma \ref{lem:sparse_measure_con} to \eqref{eq:low_2}, we have
\begin{align*}
|(\mathcal{F}_{\Z}^{-1}(\tilde{a}_{\tau})* f)(x)|
\lesssim &\sum_{|y|\leq 4\tau^2} \frac{J}{\tau^2} \sum_{q=1}^J q^{-1} |H(q,y)|\cdot |f(x-y)|\\
&+\sum_{|y|>4\tau^2} \frac{J \tau^2}{y^2} \sum_{q=1}^J q^{-1} |H(q,y)|\cdot |f(x-y)|\\
\lesssim & J \Bigl\|\sum_{q=1}^J q^{-1} |H(q,\cdot)| \Bigr\|_{\ell^{\infty}}\cdot \frac{1}{\tau^2} \sum_{|y|\leq  4\tau^2} |f(x-y)|\\
&\qquad+J \Bigl\|\sum_{q=1}^J q^{-1} |H(q,\cdot)| \Bigr\|_{\ell^\infty} \cdot \sum_{|y|>4\tau^2} \frac{\tau^2}{y^2} |f(x-y)|
\end{align*}
Lemma \ref{lem:low_goal} implies $\|\sum_{q=1}^J q^{-1} |H(q,\cdot)| \|_{\ell^\infty}\lesssim (\log J)^2$, hence
\begin{align}\label{eq:P1+P2}
|(\mathcal{F}_{\Z}^{-1}(\tilde{a}_{\tau})* f)(x)|\lesssim 
J (\log J)^2 \left( \frac{1}{\tau^2} \sum_{|y|\leq  4\tau^2} |f(x-y)|+\sum_{|y|>4\tau^2} \frac{\tau^2}{y^2} |f(x-y)| \right)
\end{align}
The admissibility of $\tau$ implies that there exists good intervals $I_k\ni x$ such that $|I_k|=2^k \tau^2$, $k\geq 2$. 
Hence we can estimate the first sum on the right-hand-side of \eqref{eq:P1+P2} as follows.
\begin{align}\label{eq:P1}
\frac{1}{\tau^2} \sum_{|y|\leq  4\tau^2}\lesssim \frac{1}{|3I_2|}\sum_{y\in 3I_2} |f(y)|= \la f\ra_{3I_2,1}\lesssim \la f\ra_{2E,1}.
\end{align}
For the second sum of the right-hand-side of \eqref{eq:P1+P2}, we have
\begin{align}\label{eq:P2}
\sum_{|y|>4\tau^2} \frac{\tau^2}{y^2} |f(x-y)|
=&\sum_{k=3}^{\infty}\,\, \sum_{2^{k-1}\tau^2<|y|\leq 2^k \tau^2}  \frac{\tau^2}{y^2} |f(x-y)| \notag\\
\lesssim &\sum_{k=3}^{\infty} 2^{-k} \,\, \frac{1}{2^k \tau^2} \sum_{2^{k-1}\tau^2<|y|\leq 2^k \tau^2} |f(x-y)| \notag\\
\lesssim &\sum_{k=3}^{\infty} 2^{-k} \, \la f\ra_{3I_k,1}  
\lesssim \la f\ra_{2E,1}.
\end{align}
Combining \eqref{eq:P1+P2} with \eqref{eq:P1}, \eqref{eq:P2}, Lemma \ref{lem:tilde_a_tau_Low} is proved. 
This also completes the proof of Lemma \ref{lem:sparse_key}. 
\end{proof}

\section{Complements}  \label{s:comp}

\subsection{The Square Integers}
The sparse bound has notable consequences for the maximal operator $A$.  One set of inequalities are weighted inequalities, 
for weights in appropriate Muckenhoupt classes.  These properties, with quantitative bounds, are well known consequences. See 
for instance the main theorem of \cite{MR3897012}.  Similarly, vector valued inequalities follow. From the note \cite{170909647C}, we have 

\begin{corollary}  For the maximal operator $A$, and $ 3/2< p \leq  \infty$, we have for a sequence of non-negative functions $ (f_j) $ defined on the integers, 
there holds 
\begin{equation} 
\Biggl\|   \Biggl[ \sum_j (A f_j)^p \Biggr]^{1/p} \Biggr\| _{ \ell ^2}  
\lesssim 
\Biggl\|   \Biggl[ \sum_j ( f_j)^p \Biggr]^{1/p} \Biggr\| _ { \ell ^2}. 
\end{equation} 
\end{corollary} 
The inequalities above are trivial for $ 2\leq p \leq \infty $. Otherwise, these are new inequalities, moreover they self-improve to weighted inequalities 
in the same range of $ p$.  

This contrasts with the main result of \cite{2015arXiv151207518M}, which imply for instance 
\begin{equation} 
\Biggl\|   \Biggl[ \sum_j (A f_j)^2 \Biggr]^{1/2} \Biggr\| _{ \ell ^p}  
\lesssim 
\Biggl\|   \Biggl[ \sum_j ( f_j)^2 \Biggr]^{1/2} \Biggr\| _ { \ell ^p}, \qquad  1< p < \infty . 
\end{equation}

As mentioned, the $\ell^p$ improving inequality is sharp, up to the end point.  Let $f$ be the indicator of the first $N$ square integers, and 
$g = \delta_0$.  Then, for $I=[0,N^2]$, we have 
\begin{align}
N^{-2}\leq N^{-2} (A_N f, g) \lesssim \langle f \rangle_{I, p} \langle g \rangle_{I, p} = N^{-3/p}. 
\end{align}

Endpoint $ L ^{p}$-improving estimates are the strongest form of these inequalities.  Since our result is sharp in the index $ p$, it is noteworthy 
that the proof delivers a Orlicz type endpoint estimate.  Keep track of the logarithms in \eqref{e:loglog}, and repeat the argument in 
\eqref{e:loglog2}. We see this strengthening of 
Theorem~\ref{thm:improving_weak}: 
for any interval $I$ with length $N^2$,  the inequality below 
holds for any indicator functions $f=\chi_F$ supported on $2I$ and $g=\chi_G$ supported on $I$.  
\begin{align} \label{e:Orlicz}
(A_N f, g) \lesssim    \psi (\langle f \rangle _{2I,1} ) \psi (   \langle g \rangle _{I,1} ) 
 \lvert  I\rvert  . 
\end{align}
Here $ \psi (x) = x ^{2/3} (1+ \log \lvert  x\rvert ) ^{4/3} $.  
This is a restricted weak type estimate from $ L ^{3/2,1}(\log L) ^{4/3}$ to $ L ^{3, \infty }(\log L) ^{-4/3} $. 
It would be very interesting if the powers of the logarithm were sharp, although we have no idea how such an argument would proceed. 
Our proof gives a similar refinement of the sparse bound, see \eqref{eq:HL_sparse}. 

Returning to the sharpness, we can now give a logarithmic refinement. 
No set that is `half-dimensional' can have a  `full intersection' with many translates of the square integers.  

\begin{proposition}  For all $ 0< \epsilon < 1$ and integers $N$, sets $G\subset [0,N^2] $ of 
cardinality $N$, there holds 
\begin{equation}
\epsilon^{3}  | \{ A_{ N} \chi_G >  \epsilon \} |  \lesssim    (\log N) ^{8}.   
\end{equation}
\end{proposition} 

\begin{proof} 
Let $H = \{ A_{ N} \chi_G >  \epsilon \} $ and $I=[0,N^2]$.  We have from \eqref{e:Orlicz}, 
\begin{align} 
\epsilon \langle \chi_H \rangle_{I,1}  \leq |I|^{-1} ( A_{ N} \chi_G , \chi_H )  \leq   \psi (\langle \chi_G \rangle_{I,1} ) \psi (\langle \chi_H \rangle_{I,1} )   
 \lesssim  N  ^{- 2/3} (\log N) ^{8/3} \langle \chi  _H \rangle_{I,1} ^{2/3}.  
\end{align} 
  This implies our proposition.  
\end{proof}

A final remark on the square integers concerns the continuous analog, which is convolution with respect to the measure $ h (x) = x ^{-1/2} \mathbf 1_{[0,1]} (x)$. This function appeared already in \eqref{eq:gammaN_Fourier_inv}. 
The sharp exponent in this case, $ p= 4/3$,  is entirely different from the discrete case. 
It is a classical fact that for functions $ \phi $ supported on $I= [0,1]$, we have 
\begin{equation}
 \langle  h \ast \phi  \rangle _{I,4} \lesssim  \langle \phi  \rangle _{I, 4/3}.  
\end{equation}
Here, we are adapting our notation to the continuous case. 
This is sharp, as seen by taking $ \phi = \mathbf 1_{[0, \delta )}$, for $ 0< \delta < 1$.   The arguments of Littman \cite{MR0358443} and 
Strichartz \cite{MR0256219} apply, since the Fourier transform of $ \gamma $ is given in terms of  Bessel function. 
One can then apply their analytic interpolation argument.   If the restricted weak type variant of the inequality above is enough, then  the High Low method quickly supplies a proof.  

\subsection{Other Averages}
There is a general conjecture that one can make, concerning $ \ell ^{p}$ improving estimates for averages over more general arithmetic sequences.  
Below, we stipulate an improving estimate that is only a function of the degree of the polynomial in question.

\begin{conjecture}\label{j:poly}  
For all integers $ d \geq 2 $, there is an $  1< q = q_d <2$ so that for any polynomial $ p (x)$ of degree $ d$, mapping the integers to the integers, 
 the following inequality holds uniformly in integers $ N\geq 1$: 
Set
\begin{equation*}
A_N f (x) = \frac1{N} \sum_{n=1} ^{N} f (x+ p (n)). 
\end{equation*} 
For an interval $ I =[0, N^d] $, and function $ f$  supported on $ I$, there holds 
\begin{equation*}
 \langle  A_N f  \rangle _{I, q' } \lesssim  \langle  f \rangle _{I,q}. 
\end{equation*}
\end{conjecture}

Dimensional considerations show that $q_d  = 2 - 1/d$ would be optimal. 
And, there are some supporting results, namely  \cites{MR3892403,MR3933540}, which concern Hilbert transforms.  
Generalizations of these arguments 
suggest that the best result   one can hope for is exponentially worse than the best possible 
bound, namely $ 2-q_d \simeq 2 ^{-d}$. 
(An important obstruction arises from the so-called minor arcs.)
In light of this, perhaps one can restrict attention to the case of $ d=2$.  
\begin{equation*}
\textup{In the case of  degree $ d=2$ in Conjecture~\ref{j:poly}, can one take $ 3/2< q < 2$?}
\end{equation*}
We don't know the answer even if one further specializes to the second degree polynomial $ p (x) = x^2 +x$. 
This highlights how strongly our argument depends upon the remarkable result of \cite{MR0563894}.  

\bigskip 

In light of the discussion above,  a   open-ended   question comes to mind: Are there other arithmetic type averaging operators for which there is a strong parallel between the continuous and discrete theories of improving estimates?  Our current examples concerning the square integers, and the spherical averages, in the fixed radius and maximal variants, indicate that a positive answer depends upon a delicate analysis of cyclic variants of the averages in question.

\appendix 
\section{Proof of Lemma \ref{lem:H_multi}}\label{app:mul}
\begin{proof}
Expanding $H_1(q_1, x) H_1(q_2,x)$, we have
\begin{align}\label{eq:app1}
H(q_1,x)H(q_2,x)
=\sum_{\substack{a_1=1 \\ (a_1,q_1)=1}}^{q_1} \sum_{\substack{a_2=1 \\ (a_2,q_2)=1}}^{q_2} G(a_1, q_1) G(a_2, q_2)e((a_1q_2+a_2q_1)x/q_1q_2)
\end{align}
Observe that 
\begin{align}\label{eq:app2}
G(a_1,q_1)G(a_2,q_2)=\varepsilon_{q_1}\varepsilon_{q_2}\varepsilon_{q_1q_2}^{-1}\left(\frac{q_1}{q_2}\right)\left(\frac{q_2}{q_1}\right)G(a_1q_2+a_2q_1, q_1q_2).
\end{align}
Indeed, 
\begin{align*}
G(a_1,q_1)G(a_2,q_2)
=&\varepsilon_{q_1}\varepsilon_{q_2} \sqrt{q_1q_2} \left(\frac{a_1}{q_1}\right) \left(\frac{a_2}{q_2}\right)\\
=&\varepsilon_{q_1}\varepsilon_{q_2} \sqrt{q_1q_2} \left(\frac{a_1q_2}{q_1}\right)\left(\frac{a_2q_1}{q_2}\right) \left(\frac{q_1}{q_2}\right)\left(\frac{q_2}{q_1}\right)\\
=&\varepsilon_{q_1}\varepsilon_{q_2}  \left(\frac{q_1}{q_2}\right)\left(\frac{q_2}{q_1}\right) \sqrt{q_1q_2} \left(\frac{a_1q_2+a_2q_1}{q_1}\right)\left(\frac{a_1q_2+a_2q_1}{q_2}\right)\\
=&\varepsilon_{q_1}\varepsilon_{q_2} \left(\frac{q_1}{q_2}\right)\left(\frac{q_2}{q_1}\right) \sqrt{q_1q_2} \left(\frac{a_1q_2+a_2q_1}{q_1q_2}\right)\\
=&\varepsilon_{q_1}\varepsilon_{q_2}\varepsilon_{q_1q_2}^{-1}\left(\frac{q_1}{q_2}\right)\left(\frac{q_2}{q_1}\right) G(a_1q_2+a_2q_1,q_1q_2).
\end{align*}
Hence by \eqref{eq:app1} and \eqref{eq:app2}, we have
\begin{align}\label{eq:app3}
|H_1(q_1,x)H_1(q_2,x)|
=&\left|\sum_{\substack{a_1=1 \\ (a_1,q_1)=1}}^{q_1} \sum_{\substack{a_2=1 \\ (a_2,q_2)=1}}^{q_2} G(a_1q_2+a_2q_1, q_1q_2)e((a_1q_2+a_2q_1)x/q_1q_2)\right| \notag\\
=&\left|\sum_{\substack{a=1\\ (a,q_1q_2)=1}}^{q_1q_2} G(a,q_1q_2) e(ax/q_1q_2)\right|\\
=&|H_1(q_1q_2,x)|. \notag
\end{align}
The reason behind \eqref{eq:app3} concerns the multipicative groups $ \mathbb Z _q ^{\ast} $.  One can construct a map $\tau$ from $\Z_{q_1}^\ast \times \Z_{q_2}^\ast$ to $\Z_{q_1q_2}^\ast$, defined by 
\begin{align*}
\tau(a_1,a_2)=a_1q_2+a_2q_1.
\end{align*}
One easily checks this map $\tau$ is well-defined since $(\tau(a_1,a_2),q_j)=(a_j,q_j)=1$ for $j=1,2$.
This map is injective since $\tau(a_1,a_2)=\tau(a_1',a_2')$ would imply $a_j=a_j'$ for $j=1,2$.
This map is also subjective since $|\Z_{q_1}^\ast\times \Z_{q_2}^\ast|=\varphi(p_1)\varphi(p_2)=\varphi(p_1p_2)=|\Z_{p_1p_2}^\ast|$, where $\varphi$ is the Euler's phi function, and we used the multiplicative property of $\varphi$ here. 
Hence $\tau$ is bijective, and \eqref{eq:app3} is verified.
This finishes the proof of Lemma \ref{lem:H_multi}.
\end{proof}

\bibliographystyle{amsplain}	


\end{document}